\documentclass[11pt,reqno]{amsart}
\usepackage[margin=1.1in]{geometry}

\usepackage{amsmath,amsthm,amssymb}
\usepackage{mathrsfs} 
\usepackage{mathabx}
\usepackage{array} 
 
\usepackage{enumitem}

\usepackage{pgf,tikz,pgfplots}

\usetikzlibrary{arrows}
\usepackage[unicode]{hyperref}

\usepackage{bm, bbm} 
\usepackage{amsmath, amsfonts, amssymb, amsthm} 
\usepackage{mathrsfs}
\usepackage{graphicx} 
\usepackage{float} 
\usepackage{enumitem} 
\usepackage{fancyhdr} 
\usepackage{caption}  
\usepackage{booktabs} 
\usepackage{tikz}
\usepackage[comma,sort&compress, numbers]{natbib}
\usepackage{subcaption}

\numberwithin{equation}{section} 

\newtheorem{thm}{Theorem}[section]
\newtheorem{corollary}[thm]{Corollary}
\newtheorem{prop}[thm]{Proposition}
\newtheorem{lem}[thm]{Lemma} 
\theoremstyle{definition} 
\newtheorem{defn}[thm]{Definition}
\newtheorem{remark}[thm]{Remark}
\newtheorem{assumption}[thm]{Assumption} 

\def\R{\mathbb{R}}

\def\P{\mathbb{P}}
\def\E{\mathbb{E}}

\def\CF{\mathcal{F}}

\def\Var{\mathrm{Var}}
\def\Cov{\mathrm{Cov}}

\def\Pto{\stackrel{P}{\rightarrow}}
\def\dto{\stackrel{D}{\rightarrow}}

\def\1{\mathbbm{1}}

\def\Aut{\mathrm{Aut}}

\allowdisplaybreaks

\begin{document}

\title[ Subgraph Counts in Sparse Inhomogeneous Random Graphs ]{ Asymptotic Normality of Subgraph Counts in Sparse Inhomogeneous Random Graphs }

\author[Chatterjee]{Sayak Chatterjee} 
\address{Department of Statistics and Data Science, University of Pennsylvania, Philadelphia, USA} \email{sayakc@wharton.upenn.edu}

\author[Chatterjee]{Anirban Chatterjee} 
\address{Department of Statistics, University of Chicago, USA} \email{anirbanc@uchicago.edu}

\author[Chakraborty]{Abhinav Chakraborty}
\address{Department of Statistics, Columbia University, New York, USA} 
\email{ac4662@columbia.edu}

\author[Bhattacharya]{Bhaswar B. Bhattacharya} 
\address{Department of Statistics and Data Science, University of Pennsylvania, Philadelphia, USA} \email{bhaswar@wharton.upenn.edu}

\begin{abstract}
In this paper, we derive the asymptotic distribution of the number of copies of a fixed graph $H$ in a random graph $G_n$ sampled from a sparse graphon model. Specifically, we provide a refined analysis that separates the contributions of edge randomness and vertex-label randomness, allowing us to identify distinct sparsity regimes in which each component dominates or both contribute jointly to the fluctuations. As a result, we establish asymptotic normality for the count of any fixed graph $H$ in $G_n$ across the entire range of sparsity (above the containment threshold for $H$ in $G_n$). These results provide a complete description of subgraph count fluctuations in sparse inhomogeneous networks, closing several gaps in the existing literature that were limited to specific motifs or suboptimal sparsity assumptions. 
\end{abstract}

\maketitle

\section{Introduction} \label{sec:intro}

A \textit{graphon} is a symmetric measurable function $W : [0,1]^2 \to [0,1]$, that is, $W(x,y) = W(y,x)$ for all $x,y \in [0,1]$. Graphons arise as limit objects of sequences of large graphs and have attracted significant attention in recent years. They form a powerful bridge between combinatorics and analysis and have found applications across a range of disciplines, including statistical physics, probability, and statistics; see, for example, \cite{borgs2008convergent,borgs2012convergent,bmpotts,chatterjee2013estimating,chatterjee2011large}. For a comprehensive exposition of the theory of graph limits, we refer to \cite{lovasz2012large}. Graphons also provide a natural probabilistic model for generating inhomogeneous random graphs, which generalize the classical (homogeneous) Erd\H{o}s--R\'enyi random graph, where edge probabilities vary across vertex pairs according to the underlying graphon 
(see
\cite{bollobas2007phase,diaconis1981statistics,lovasz2006limits,boguna2003class} among others).

Random graphs generated from a graphon model produces dense graphs, which have $O(n^2)$ edges with high probability. In contrast, many real-world networks, arising from social, biological, and technological applications, are observed to be sparse. To model such networks, the \emph{sparse graphon model} has been proposed, which introduces a parameter $\rho_n$ to control the edge density~(see \cite{bickel2011method,bickel2009nonparametric,klopp2017oracle,lunde2023subsampling} references therein):

\begin{defn}[Sparse graphon model] Given a sparsity parameter $\rho_n \in (0, 1)$ and a graphon $W$ with $\int_{[0, 1]^2} W(x, y) \mathrm d x \mathrm dy > 0$, the {\it $W$-random graph with sparsity $\rho_n$} on the vertex set $[n] := \{1, 2, \ldots, n\}$, hereafter denoted by $G(n, \rho_n, W)$, is obtained by connecting the vertices $i$ and $j$ with probability $\rho_n W(U_i, U_j)$, independently for all $1\leq i < j \leq n$, where $\{U_{i}: 1 \leq i \leq n\}$ is a sequence of i.i.d. $U([0, 1])$ random variables. 
\label{defn:W} 
\end{defn}

The model in Definition \ref{defn:W} encompasses many well-known network models, including the classical Erd\H{o}s-R\'{e}nyi random graph model (where $W$ is a constant function), the stochastic block model \citep{bickel2009nonparametric,holland1983stochastic} (where 
$W$ is a block function), latent space models \citep{hoff2002latent}, random dot-product graphs \citep{dotproduct2018statistical,rubin2022statistical}, and random geometric graphs \citep{penrose2003}, among others. Over the years, the sparse graphon model has emerged as a fundamental tool in modern network analysis, with applications in community detection, subgraph count statistics, and nonparametric estimation of graph parameters, among others.

Among the basic features of a network are its subgraph (motif) counts, that is, the frequencies of particular patterns (subgraphs), such as the number or density of edges, triangles, or stars within the network. 
These encode important structural information about the geometry of a network, and many features of practical interest can be derived from motif counts, such as the clustering coefficient \citep{watts1998collective}, degree distribution \citep{networkdegree}, and transitivity \citep{holland1971transitivity} (see \cite{rubinov2010complex} for others).  In the framework of the graphon model, subgraph counts are often referred to as network moments \citep{bickel2011method,borgs2010moments}, which provide an important tool for inferring properties of the underlying graphon  (see \cite{subsamplingnetwork,lin2020bootstrap,shao2022higher,green2022bootstrapping,lunde2023subsampling,maugis2020testing} and references therein). Consequently, understanding the asymptotic properties of subgraph counts in $W$-random graphs is a problem of central importance in statistical network analysis. To this end, given a fixed graph $H=(V(H), E(H))$ denote by $X_{n}(H, G_n)$ the number of copies of $H$ in the $W$-random graph $G_n\sim G(n, \rho_n, W)$. More formally,  
\begin{align}\label{eq:XHW}
X_{n}(H, G_n)=\sum_{1\leq i_{1}<\cdots<i_{|V(H)|}\leq n}\sum_{H'\in
  \mathscr{G}_H(\{i_{1},\ldots, i_{|V(H)|} \}) } 
\prod_{(i_{s}, i_{t}) \in  E(H')} a_{i_{s}i_{t}} , 
\end{align}
where, $a_{ij} = 1$, if $(i, j)$ is an edge in $G_n$ and 0 otherwise, for $1 \leq i < j \leq n$, and, for any set $S \subseteq [n]$, $\mathscr G_H(S)$ denotes the
collection of all subgraphs of the complete graph $K_{|S|}$ on the vertex
set $S$ which are isomorphic to $H$.\footnote{Note that we count unlabelled copies of $H$. Several other papers count labelled copies, where $X_n(H, G_n)$ is  multiplied by $|\text{Aut}(H)|$.} 
Note that
\begin{equation}
    | \mathscr{G}_H(\{1,\ldots, |V(H)| \}) | = \dfrac{|V(H)|!}{|\text{Aut}(H)|}, 
    \label{eq:G_H}
\end{equation}
where $\text{Aut}(H)$ is the set of all automorphisms of the graph $H$, that is, the collection of permutations of the vertex set $V (H)$ such that $(x, y) \in E(H)$ if and only if $(\sigma(x), \sigma(y)) \in E(H)$. Hence, 
\begin{align}\label{eq:EXHW}
\mathbb E[X_n(H, G_n)] & =\sum_{1\leq i_{1}<\cdots<i_{|V(H)|}\leq n}\sum_{H'\in \mathscr{G}_H(\{i_{1},\ldots, i_{|V(H)|} \}) }\E\left[ \prod_{(i_{s}, i_{t}) \in E(H')} [\rho_nW(U_{i_{s}}U_{i_{t}}) ]\right] \nonumber \\ 
&= \dfrac{(n)_{|V(H)|}}{|\text{Aut}(H)|} \rho_n^{|E(H)|} t(H,W) , 
\end{align}
where $(n)_{|V(H)|}:=n(n-1)\cdots (n-|V(H)|+1)$ and 
\begin{align}\label{eq:tHW}
t(H,W)=\int_{[0,1]^{|V(H)|}}\prod_{(a, b) \in E(H)}W(x_{a},x_{b})\prod_{a=1}^{|V(H)|}\mathrm{d} x_{a}  
\end{align}
is the {\it homomorphism density} of the graph $H$ in the graphon $W$.

Classical results of \citet{rucinski1988small,barbour1989central,nowicki1989asymptotic,janson1990functional} provide a complete characterization of the asymptotic normality of $X_n(H, G_n)$ (after appropriate centering and scaling) when $G_n \sim G(n, \rho_n)$ is the (homogeneous) Erd\H{o}s-R\'{e}nyi model (see \citet[Chapter 6]{JLR} for comprehensive treatment). For $G_n \sim G(n, \rho_n, W)$ sampled from the sparse random graphon model, the fluctuation of $X_n(H, G_n)$ was first studied in the seminal paper of \citet{bickel2011method}. This result has been refined and extended in many directions (see the discussion  Section~\ref{sec:Wsparse}). However, existing results remain incomplete: they apply only to specific types of subgraphs and often do not cover the full sparsity regime. Our objective in this paper is to fill these missing gaps. Specifically, we establish the asymptotic normality of $Z_n(H, G_n),$ the standardized version of $X_n(H,G_n),$ for all fixed graphs $H$ in the entire allowable sparsity regime (that is, whenever $\rho_n$ is above the containment threshold of $H$ in $G(n, \rho_n, W)$). In fact, we provide a more fine-grained analysis that decouples the randomness of the vertex labels from the randomness of the edges and identifies distinct regimes of sparsity where the two components contribute, either separately or simultaneously, to the fluctuations of $Z_n(H, G_n)$. This relies on decomposing the centered subgraph count as follows: 
\begin{align}
    \Delta(H, G_n) & := X_n(H, G_n) - \mathbb{E}[X_n(H, G_n)] \notag \\
    &= \big(X_n(H, G_n) - \mathbb{E}[X_n(H, G_n) \mid \mathcal{F}(\bm{U}_n) ] \big) + \big(\mathbb{E}[X_n(H, G_n) \mid \mathcal{F}(\bm{U}_n) ] - \mathbb{E}[X_n(H, G_n)] \big) \notag \\
    &=: \Delta_1(H, G_n) + \Delta_2(H, G_n) , 
    \label{eq:XHGn12}
\end{align} 
where $\mathcal{F}(\bm{U}_n)$ is the sigma-algebra generated by the collection $\bm U_n= \{U_1, U_2, \ldots, U_n\}$. Note that, conditioned on $\bm U_n$, the randomness in the first term $\Delta_1(H, G_n)$ arises solely from the edges of $G_n$. On the other hand, $\Delta_2(H, G_n)$ depends only on vertex labels $\{U_1, U_2, \ldots, U_n\}$, that is, it is $\mathcal{F}(\bm{U}_n)$-measurable. 
Contributions of $\Delta_1(H, G_n)$ and $\Delta_2(H, G_n)$ to the total variance of $X_n(H, G_n)$ depend on the regime of sparsity and also on a notion of $H$-regularity of the graphon $W$ (see Definition \ref{regular}). Specifically, we obtain the following results:  

\begin{itemize}

\item If $W$ is $H$-regular, the contribution of $\Delta_2(H, G_n)$ to the variance of $X_n(H, G_n)$ is asymptotically negligible, and the limiting distribution of $X_n(H, G_n)$ is governed solely by $\Delta_1(H, G_n)$ (see Theorem \ref{thm:regular}).

\item If $W$ is not $H$-regular, the asymptotic behaviors of $\Delta_1(H, G_n)$ and $\Delta_2(H, G_n)$ exhibits a phase transition. Below a critical threshold, $\Delta_2(H, G_n)$ contributes negligibly to the variance of $X_n(H, G_n)$, and the fluctuations are driven entirely by $\Delta_1(H, G_n)$. Above the threshold, the roles reverse, with $\Delta_1(H, G_n)$ becoming negligible and $\Delta_2(H, G_n)$ governing the fluctuations. At the threshold, both terms contribute, and the appropriately standardized pair $(\Delta_1(H, G_n), \Delta_2(H, G_n))$ converges jointly to independent normal distributions with nontrivial variances (see Theorem \ref{thm:irregular}).

\end{itemize}
The above results combined show that for any subgraph $H= (V(H), E(H))$, with $|E(H)| \geq 1$, $Z_n(H, G_n)$ is asymptotically normal whenever $\rho_n$ is above the containment threshold of $H$ in $G(n, W)$ (see Corollary \ref{cor:ZHGn}). This completes the program, initiated in \citet{bickel2011method}, of deriving the asymptotic distribution of subgraph counts (network moments) in sparse graphon models in full generality. In particular, it also answers a question posed by \citet{hladky2021limit} about the distribution of the number of cliques. As a byproduct of our analysis we obtain a conditional central limit theorem  for $\Delta_1(H, G_n)$ (in the Wasserstein distance) that might be of independent interest (see Proposition \ref{prop:main}).

\subsection*{Asymptotic Notations}

Throughout the paper we will use the following asymptotic notations: For two sequences $a_n$ and $b_n$ we will write $a_n \lesssim b_n$ and $a_n = O(b_n)$ if for all $n$ large enough $a_n \leq C _1 b_n$, for some constant $C_1 > 0$. Similarly, $a_n \gtrsim b_n$ will mean  $a_n \geq C_2 b_n.$ Also, $a_n \asymp b_n$ and $a_n=\Theta(b_n)$ will mean $C_2 b_n \leq a_n \leq C_1 b_n$, for $n$ large enough and constants $C_1, C_2 > 0$. Subscripts in the above notation, for example $\lesssim_{\square}$, $\gtrsim_{\square}$, $O_\square$, and $\asymp_{\square}$ denote that the hidden constants may depend on the subscripted parameters. Moreover, $a_n \ll b_n$, $a_n\gg b_n$, and $a_n \sim b_n$ will mean $a_n/b_n \rightarrow 0$, $a_n/b_n \rightarrow \infty$, $a_n/b_n \rightarrow 1$, as $n \rightarrow \infty$.

\section{Fluctuations of Subgraph Counts in Sparse $W$-Random Graphs}

In this section we will state our main result on the asymptotic distribution $Z_n(H, G_n)$, where $G_n \sim G(n, \rho_n, W)$ is a $W$-random graph with sparsity $\rho_n$. 


\subsection{Preliminaries}

The preliminary step towards deriving the asymptotic distribution of $Z_n(H, G_n)$ is to figure out the rate at which $\rho_n$ has to diverge such that $G_n \sim G(n, \rho_n, W)$ contains at least one (or infinitely many) copies of $H$. For the classical Erd\H{o}s-R\'{e}nyi model this is standard fare in the random graphs literature (see \citet[Chapter 3]{JLR}). In particular, it depends on whether or not the graph $H$ is balanced: 


\begin{defn}[Balanced graph]\cite[Equation (3.6)]{JLR}
    For a fixed graph $H=(V(H), E(H))$, define
    \begin{align}\label{eq:mH}
    m(H):=\max_{F\subseteq H,\,|V(F)|\ge1}\frac{|E(F)|}{|V(F)|} , 
    \end{align}
    where the maximum is taken over all possible non-empty subgraphs of $H.$ If the maximum is attained by $H$ itself, that is, $m(H)=\frac{|E(H)|}{|V(H)|}$, then $H$ said to be a {\it balanced} graph. Otherwise, $H$ said to be an {\it unbalanced} graph. Moreover, if the maximum in \eqref{eq:mH} is uniquely attained by $H$, then $H$ is called a {\it strictly balanced} graph.
    \label{balanced}
\end{defn} 

Examples of strictly balanced graphs include trees, cycles, and complete graphs. Figure \ref{fig:mH} shows examples for the other cases: 

\begin{itemize} 

\item Figure \ref{fig:mH} (a) is an example of a graph $H$ which is balanced, but not strictly balanced. Here, \eqref{eq:mH} is attained both at the full graph $H$ as well as the subgraph $F$ marked by the red dotted circle. 
 
\item Figure \ref{fig:mH} (b) is an example of an unbalanced graph $H$. Here, for the subgraph $F$ marked by the red dotted circle, $\frac{|E(F)|}{|V(F)|}=\frac{5}{4}> \frac{6}{5} = \frac{|E(H)|}{|V(H)|}$.  


\end{itemize}

\begin{figure}
    \centering
    \begin{subfigure}[b]{6.5cm}
    \centering
    \begin{tikzpicture}[scale=0.75]
    \draw[black,thick] (0,0) -- (-1.732,1);
    \draw[black,thick] (0,0) -- (-1.732,-1);
    \draw[black,thick] (-1.732,1) -- (-1.732,-1);
    \draw[black,thick] (0,0) -- (2,0);
    \filldraw[blue] (0,0) circle (2pt) node {};
    \filldraw[blue] (-1.732,-1) circle (2pt) node {};
    \filldraw[blue] (-1.732,1) circle (2pt) node {};
    \filldraw[blue] (2,0) circle (2pt) node {};
    \draw[red,thick,dashed] (-1,0) circle (1.7cm);
    \draw (-3,-0.5) node[anchor=north] {$F$};
    \end{tikzpicture} \\ 
    (a) 
    \end{subfigure}%
    \begin{subfigure}[b]{6.5cm}
    \centering
    \begin{tikzpicture}[scale=0.9]
    \draw[black,thick] (0,0) -- (-1,1);
    \draw[black,thick] (0,0) -- (-1,-1);
    \draw[black,thick] (-1,1) -- (-1,-1);
    \draw[black,thick] (-1,1) -- (-2,0);
    \draw[black,thick] (-1,-1) -- (-2,0);
    \draw[black,thick] (0,0) -- (1.414,0);
    \filldraw[blue] (0,0) circle (2pt) node {};
    \filldraw[blue] (-1,-1) circle (2pt) node {};
    \filldraw[blue] (-1,1) circle (2pt) node {};
    \filldraw[blue] (-2,0) circle (2pt) node {};
    \filldraw[blue] (1.414,0) circle (2pt) node {};
    \draw[red,thick,dashed] (-1,0) circle (1.4cm);
    \draw (-2.6,-0.5) node[anchor=north] {$F$};
    \end{tikzpicture} \\ 
    (b)
    \end{subfigure} 
    \caption{ \small{(a) A graph which is balanced, but not strictly balanced; (b) an unbalanced graph. } } 
        \label{fig:mH}
\end{figure}
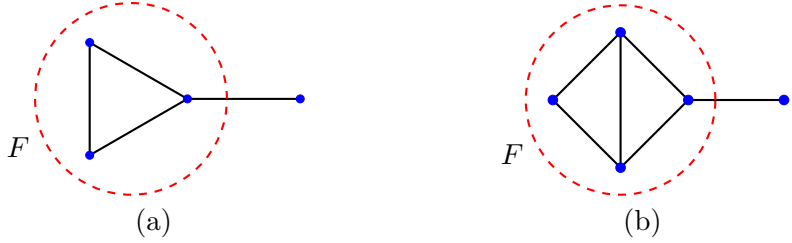

It is well-known from classical results \citep{erdos1960evolution,bollobas1981threshold} that for an Erd\H{o}s-R\'{e}nyi graph random $G(n, \rho_n)$, the threshold for the appearance of a fixed graph $H= (V(H), E(H))$ is determined by the condition $n \rho_n^{m(H)} \asymp 1$ (see \citet[Theorem 3.4]{JLR}). The proof of this result can be easily adapted to show that the threshold remains the same for the inhomogeneous random graph model $G(n, \rho_n, W)$ as well. We summarize this result in the following proposition. A proof is included in Section \ref{sec:thresholdpf} for the sake of completeness.

\begin{prop}
\label{prop:threshold} 
Fix a graph $H = (V(H), E(H))$ with $|E(H)| \geq 1$ and a graphon $W$ such that $t(H, W) > 0$. Then for $G_n \sim G(n, \rho_n, W)$ the following holds: 
    \[
   X_n(H,G_n) \Pto \left\{\begin{array}{ll}
        0  &\text{if } n \rho_n^{m(H)} \ll  1 , \\
        \infty & \text{if } n \rho_n^{m(H)} \gg 1 . 
        \end{array}\right.
  \] 
\end{prop}

\subsection{Statements of the Results} 

In light of the above proposition, hereafter, we will assume the following: 

\begin{assumption}\label{assumption} Throughout, we will assume the following on $\rho_n$, $H$, and $W$:
\begin{itemize}
\item $\rho_n \in (0, 1)$ is such that $n^{-\frac{1}{m(H)}} \ll \rho_n \ll  1$. 

\item $H = (V(H), E(H))$ is a fixed graph with $|E(H)| \geq 1.$

\item $W$ is a graphon with $t(H, W) > 0$.  
\end{itemize} 
\end{assumption}

Now, define 
\begin{align}\label{eq:ZHGn}
Z_n(H, G_n) = \frac{X_n(H, G_n) -  \E[X_n(H, G_n)]}{\sqrt{\Var[X_n(H, G_n)]}} =  \frac{\Delta(H, G_n)}{\sqrt{\Var[\Delta(H, G_n)]}} , 
\end{align}
 where $\Delta(H, G_n) := X_n(H, G_n) - \mathbb{E}[X_n(H, G_n)]$. Recalling the the decomposition in \eqref{thm:irregular}, note that 
\begin{align}
    \Var[\Delta(H, G_n)]&=\E[\Var[\Delta(H, G_n)\mid \mathcal{F}(\bm{U}_n) ]]+\Var[\E[\Delta(H, G_n)\mid \mathcal{F}(\bm{U}_n) ]]\notag\\
    &=\Var[\Delta_1(H,G_n)]+\Var[\Delta_2(H,G_n)]. 
    \label{eq:vardecomp}
\end{align} 
Our objective is to identify the regimes in $\Delta_1(H, G_n)$ and $\Delta_2(H, G_n)$ contribute to the total variance. This depends on the regime of sparsity and also on how $H$ is embedded in $W$. To this end, we need the following definition:

\begin{defn}[$H$-regularity]
A graphon $W$ is said to be $H$-regular if, for almost every $x\in[0,1]$, 
\begin{align}\label{eq:H_regular}
\bar{t}(x,H,W):=\frac{1}{|V(H)|}\sum_{a=1}^{|V(H)|} t_a(x, H, W) = t(H, W) ,  
\end{align} 
where $$t_a(x, H, W) := \E\left[\prod_{(s,t)\in E(H)}W(U_s,U_t)\mid U_{a}=x \right],$$
for $1 \leq a \leq |V(H)|.$
\label{regular}
\end{defn}

Note that in \eqref{eq:H_regular} it suffices assume that
$\overline{t}(x,H,W) $ is a constant for almost every $x \in [0, 1]$. This
is because 
\begin{align*}
\int_0^1 t_a(x, H, W)  \mathrm d x = t(H, W),  
\end{align*}
for all $a \in V(H)$. Hence, if 
$\overline{t}(x,H,W) $ is a constant almost everywhere, then the constant must be $t(H, W)$. In other words, a graphon $W$ is $H$-regular if the homomorphism density of $H$ in $W$ when one of the vertices of $H$ is marked,  is a constant independent of the value of the marked vertex. Alternatively, $H$-regularity is equivalent to the condition that the scaled $U$-statistic $\Delta_2(H,G_n)$ is first-order degenerate. 

\begin{remark}\label{RK2}
Recall that the \emph{degree function} of a graphon $W$ is defined as
\begin{align*}
d_{W}(x):=\int_0^1W(x,y) \mathrm d y .
\end{align*}
Note that when $H=K_2$ is the single edge, $$t_{1}(x,K_2,W) = \mathbb{E}\left[W(U_{1},U_{2})\bigm|U_{1}=x\right]  =
  \int_0^1W(x,y) \mathrm d y  = d_{W}(x).$$ Hence,
the notion of $K_2$-regularity
coincides with the standard notion of {\it degree regularity}, where the
 degree function $ d_{W}(x)$ is constant almost everywhere. 
\end{remark}

In the dense regime (where $\rho_n \asymp 1$), \cite{bhattacharya2021fluctuations,hladky2021limit} showed that 
$H$-regularity can result $X_n(H, G_n)$ to have non-Gaussian  fluctuations. The situation, however, is very different in the sparse regime. In this case, $H$-regularity 
implies that $\Delta_2(H, G_n)$ has asymptotically negligible contribution to the total variance of $X_n(H, G_n)$. Consequently, the asymptotic distribution of $X_n(H, G_n)$ is  determined entirely by the fluctuations of $\Delta_1(H, G_n)$.

\begin{thm}
\label{thm:regular} 
Suppose Assumption \ref{assumption} holds and the graphon $W$ is $H$-regular. Then,  
\begin{align}\label{eq:Deltaregular}
\frac{\Delta_2(H,G_n)}{\sqrt{\Var[\Delta(H, G_n)]}} \stackrel{L_2}\rightarrow 0 \quad \text{ and } \quad \frac{\Delta_1(H, G_n)}{\sqrt{\Var[\Delta(H, G_n)]}}\dto N(0,1) , 
\end{align}
Consequently,  
\begin{align}\label{eq:ZGnHregular}
Z_n(H, G_n) = \frac{\Delta(H, G_n)}{\sqrt{\Var[\Delta(H, G_n)]}}\dto N(0,1), 
\end{align}
for $Z(H, G_n)$ as defined in \eqref{eq:ZHGn}. 
\end{thm}

Next, we consider the case where $W$ is not $H$-regular. In this case, the situation is more delicate. Specifically, an additional threshold on $\rho_n$ emerges, which determines whether or not $\Delta_2(H, G_n)$ contributes to the fluctuations of $X_n(H, G_n)$. To describe this threshold we need the following definition:

\begin{defn}[Strongly balanced graph]\cite[Equation (3.17)]{JLR}
    For a fixed graph $H=(V(H), E(H))$ define
    \begin{equation}
        m_1(H):=\max_{F\subseteq H,\,|V(F)|\ge2}\frac{|E(F)|}{|V(F)|-1}, 
        \label{eq:m1H}
    \end{equation}
    where the maximum is taken over all possible subgraphs of $H$ with at least $2$ vertices. If the maximum is attained at $H$, that is, $m_1(H)=\frac{|E(H)|}{|V(H)|-1},$ then $H$ is said to be a {\it strongly balanced} graph. Moreover, if the maximum is attained uniquely at $H$, then $H$ is called a {\it strictly strongly balanced} graph. 
\end{defn}

Cycles, trees and complete graphs are strongly balanced. Figure~\ref{fig:EF} shows examples of graphs $H$ that are not strongly balanced. In both cases, for the subgraph $F$ marked by the red dotted circle, $\frac{|E(F)|}{|V(F)|-1} > \frac{|E(H)|}{|V(H)|-1}$. 


\begin{remark} 
The notion of a strongly balanced graph was introduced by \cite{rucinski1986strongly}, where it is shown that $n\rho_n^{m_1(H)} \asymp 1$ is the threshold for the containment of a rooted copy of $H$ in the Erd\H{o}s-R\'{e}nyi model $G(n, \rho_n)$. Recently, \citet{maugis2024central} showed that for the sparse graphon model the number of rooted copies of $H$ (appropriately standardized) has Gaussian fluctuations,  whenever $n\rho_n^{m_1(H)} \gg 1$. 
In Theorem \ref{thm:irregular} we prove that when $W$ is not $H$-regular, then $n\rho_n^{m_1(H)} \asymp 1$ marks the threshold where the asymptotic behavior of $\Delta_1(H, G_n)$ and $\Delta_2(H, G_n)$ undergoes a phase transition. 
\end{remark} 

In the following proposition we collect a few basic implications of being strongly balanced. The proof is given in Appendix \ref{sec:Hpf}.

\begin{prop} \label{prop:H}  For any graph $H$ with $|E(H)|\ge1$, $m(H) < m_1(H)$. Further, if $H$ is strongly balanced, then $H$ is strictly balanced.  
\end{prop}

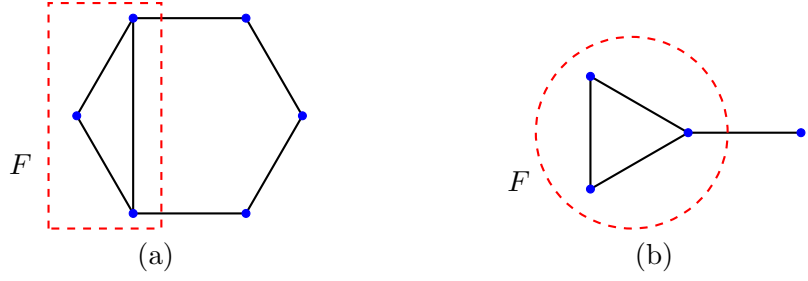
\begin{figure}
    \centering
    \begin{subfigure}[b]{6.5cm}
    \centering
    \begin{tikzpicture}[scale=0.75]
    \draw[black,thick] (0,0) -- (1,1.732);
    \draw[black,thick] (0,0) -- (1,-1.732);
    \draw[black,thick] (1,1.732) -- (3,1.732);
    \draw[black,thick] (1,1.732) -- (1,-1.732);
    \draw[black,thick] (1,-1.732) -- (3,-1.732);
    \draw[black,thick] (3,1.732) -- (4,0);
    \draw[black,thick] (3,-1.732) -- (4,0);
    \filldraw[blue] (0,0) circle (2pt) node {};
    \filldraw[blue] (1,1.732) circle (2pt) node {};
    \filldraw[blue] (1,-1.732) circle (2pt) node {};
    \filldraw[blue] (3,1.732) circle (2pt) node {};
    \filldraw[blue] (3,-1.732) circle (2pt) node {};
    \filldraw[blue] (4,0) circle (2pt) node {};
    \draw[red,thick,dashed] (-0.5,-2) rectangle (1.5,2);
    \draw (-1,-0.5) node[anchor=north] {$F$};
    \end{tikzpicture} \\
    (a) 
    \end{subfigure}%
    ~ 
    \begin{subfigure}[b]{6.5cm}
    \centering
    \begin{tikzpicture}[scale=0.75]
    \draw[black,thick] (0,0) -- (2,0);
    \draw[black,thick] (0,0) -- (-1.732,-1);
    \draw[black,thick] (0,0) -- (-1.732,1);
    \draw[black,thick] (-1.732,-1) -- (-1.732,1);
    \filldraw[blue] (0,0) circle (2pt) node {};
    \filldraw[blue] (2,0) circle (2pt) node {};
    \filldraw[blue] (-1.732,-1) circle (2pt) node {};
    \filldraw[blue] (-1.732,1) circle (2pt) node {};
    \draw[red,thick,dashed] (-1,0) circle (1.7cm);
    \draw (-3,-0.5) node[anchor=north] {$F$};
    \end{tikzpicture} \\
	(b) 
    \end{subfigure} 
       \caption{ \small{ (a) A strictly balanced graph that is not strongly balanced;  and (b) a balanced (but not strictly balanced) graph, which is not strongly balanced. } }  
        \label{fig:EF}

\end{figure}

\begin{remark} 
Proposition \ref{prop:H} shows that strongly balanced implies strictly balanced. The converse, however, is not true. Specifically, Figure \ref{fig:EF} (a)  shows a strictly balanced graph that is not strongly balanced. In this case, the subgraph $F$ marked by the red dotted rectangle satisfies
$$\frac{|E(F)|}{|V(F)|-1}=\frac{3}{2}>\frac{7}{5}=\frac{|E(H)|}{|V(H)|-1}.$$
Further Figure \ref{fig:EF} (b) is another example of graph that is balanced (but not strictly balanced), which is not strongly balanced.  
\label{remark:strongH}
\end{remark}

The following result provides a complete characterization of the fluctuations of $\Delta_1(H, G_n)$ and $\Delta_2(H, G_n)$ and, as a result, that of $\Delta(H, G_n)$, when the graphon $W$ is not $H$-regular. 

\begin{thm}
\label{thm:irregular} 
Suppose Assumption \ref{assumption} holds and the graphon $W$ is not $H$-regular. Then the following hold:
\begin{enumerate}
    \item[$(1)$] If $n\rho_n^{m_1(H)} \ll 1,$ then 
    \begin{align}\label{eq:Deltairregularb}
    \frac{\Delta_2(H,G_n)}{\sqrt{\Var[\Delta(H, G_n)]}} \stackrel{L_2}\rightarrow 0 \quad \text{ and } \quad  \frac{\Delta_1(H,G_n)}{\sqrt{\Var[\Delta(H, G_n)]}}\dto N(0,1).
    \end{align}
    
    \item[$(2)$] If $n\rho_n^{m_1(H)} \gg 1,$ then 
        \begin{align}\label{eq:Deltairregularu}
        \frac{\Delta_1(H,G_n)}{\sqrt{\Var[\Delta(H, G_n)]}} \stackrel{L_2}\rightarrow 0 \quad \text{ and } \quad  \frac{\Delta_2(H,G_n)}{\sqrt{\Var[\Delta(H, G_n)]}}\dto N(0,1). 
        \end{align}
    
    \item[$(3)$] If $n\rho_n^{m_1(H)}\to c\in(0,\infty),$ then there exists a constant $\kappa\in(0,1)$ (defined in \eqref{eq:varianceRW}) such that 
    \begin{align}\label{eq:Deltairregulart}
    \begin{pmatrix} 
    \frac{\Delta_1(H,G_n)}{\sqrt{\Var[\Delta(H, G_n)]}} \\ 
    \frac{\Delta_2(H,G_n)}{\sqrt{\Var[\Delta(H, G_n)]}}
    \end{pmatrix} \dto N \left( 
    \begin{pmatrix} 
    0 \\ 
    0
    \end{pmatrix},   \begin{pmatrix} 
    \kappa & 0 \\ 
    0 & 1- \kappa 
    \end{pmatrix} \right) . 
    \end{align}  
\end{enumerate}
Consequently, in all three cases, 
\begin{align}
\label{eq:ZGnHirregular}
Z_n(H, G_n) = \frac{\Delta(H, G_n)}{\sqrt{\Var[\Delta(H, G_n)]}}\dto N(0,1), 
\end{align}  
whenever $n^{-\frac{1}{m(H)}} \ll \rho_n \ll  1$. 
\end{thm} 

In words, Theorem \ref{thm:irregular} shows that when $W$ is not $H$-regular, then there are 3 possibilities depending on the strictly balanced coefficient $m_1(H)$: 

\begin{itemize} 

\item $n\rho_n^{m_1(H)} \ll 1$: In this case, $\Delta_2(H, G_n)$ has asymptotically negligible contribution to the total variance of $X_n(H, G_n)$ and the asymptotic distribution of $X_n(H, G_n)$ is  determined entirely by the fluctuations of $\Delta_1(H, G_n)$.

\item $n\rho_n^{m_1(H)} \gg 1$: In this case, $\Delta_1(H, G_n)$ has asymptotically negligible contribution to the total variance of $X_n(H, G_n)$ and the asymptotic distribution of $X_n(H, G_n)$ is  determined entirely by the fluctuations of $\Delta_2(H, G_n)$.

\item $n\rho_n^{m_1(H)} \rightarrow c \in (0, \infty)$: In this case, both $\Delta_1(H, G_n)$ and $\Delta_2(H, G_n)$ contribute to  the fluctuations of $X_n(H, G_n)$. In particular, $(\Delta_1(H, G_n), \Delta_2(H, G_n))$ (after appropriate standardization) converges jointly to independent normals with non-trivial variances. 

\end{itemize}

\begin{remark} 
Recall that Theorem \ref{thm:irregular} is under the standing assumption that $n \rho_n^{m(H)} \ll  1$. By Proposition, we know that \ref{prop:H} $m(H) < m_1(H)$, for every subgraph $H$ with $|E(H)|\ge1.$ Hence, all the above three regimes in Theorem \ref{thm:irregular} are relevant. 
\end{remark}

Another consequence Theorem \ref{thm:irregular} is that $Z_n(H, G_n)$ is asymptotically normal in all the three  cases. We know from Theorem \ref{thm:regular} that the same is also true when $W$ is $H$-regular (recall \eqref{eq:ZGnHregular}). Hence, we have the following corollary: 

\begin{corollary}\label{cor:ZHGn}
Suppose Assumption \ref{assumption} holds. Then the following hold: 
$$Z_n(H, G_n) = \frac{\Delta(H, G_n)}{\sqrt{\Var[\Delta(H, G_n)]}}\dto N(0,1), $$  
where $Z_n(H, G_n)$ is as defined in \eqref{eq:ZHGn}. 
\end{corollary}

This shows that $Z_n(H, G_n)$ is asymptotically normal (irrespective of whether or not $W$ is $H$-regular), for any $\rho_n \ll 1$ above the containment threshold. In previous works, this result has been established for either certain types of graph $H$ or in certain sub-regimes of sparsity above the containment threshold. Our result closes all the gaps in both these directions by showing that asymptotic normality of $Z_n(H, G_n)$ holds for all fixed graph $H$ and for the full range of  the sparsity parameter above the containment threshold. A detailed comparison with existing results is given in Section \ref{sec:Wsparseresults}.

\subsection{A General Conditional CLT and Proof Outline}

The proofs of Theorem \ref{thm:regular} and Theorem \ref{thm:irregular} rely on a result about the conditional fluctuations of $\Delta_1(H,G_n)$, which might be of independent interest. To state the result we need the following definition: 

\begin{defn}
Given two real-valued random variables $X$ and $Z$ and a $\sigma$-algebra $\mathcal{F}$, the Wasserstein distance between $X \mid \mathcal F$ (the conditional law of $X$ given the $\sigma$-algebra $\mathcal{F}$) and $Z$ is defined as
$$d_{\mathrm{Wass}}(X\mid\mathcal{F},Z):=\sup_{g \in\mathcal{L}} \left|\E[g(X)\mid\mathcal{F}]-\E[g(Z)] \right|,$$
where $\mathcal L$ is the collection of 1-Lipschitz functions from $\R \rightarrow \R$.  
\label{defn:cond_wass}
\end{defn}

We are now ready to state our result on the conditional convergence of $\Delta_1(H,G_n)$ (appropriately normalized). Recall that $\mathcal F(\bm{U}_n)$ is the sigma-algebra generated by the collection $\bm{U}_n = \{U_1, U_2, \ldots, U_n\}$.

\begin{prop} 
Suppose Assumption \ref{assumption} holds. Then, for $Z \sim N(0, 1)$, 
    \begin{align}\label{eq:conditionaldistance}
    d_{\mathrm{Wass}} \left( \frac{\Delta_1(H,G_n)}{\sqrt{\Var[\Delta_1(H, G_n) \mid \mathcal{F}(\bm U_n) ]}}\mid \mathcal{F}(\bm{U}_n) , Z  \right)  \Pto 0. 
    \end{align} 
        \label{prop:main} 
\end{prop}

The proof of Proposition \ref{prop:main} is given in Section \ref{sec:conditionalpf}. The proof uses on Stein's method based on dependency graphs, which bounds the distance to normality in the Wasserstein distance. Note that the Wasserstein distance in the LHS of \eqref{eq:conditionaldistance} is itself random (since it is defined conditional on $\mathcal{F}(\bm{U}_n)$), hence the convergence in \eqref{eq:conditionaldistance} is shown to hold in probability. An important feature of result in Proposition \ref{prop:main} is that it holds regardless of whether or not the graphon $W$ is $H$-regular. Hence, Proposition \ref{prop:main} can be used to prove the assertions of Theorem \ref{thm:regular} and Theorem \ref{thm:irregular} as follows: 

\begin{itemize} 

\item {\it $W$ is $H$-regular} (Theorem \ref{thm:regular}): This is proved in Section \ref{sec:thm1proof}. The first step is a variance computation showing that 
$$\Var[\Delta_2(H,G_n)] \ll \Var[\Delta(H, G_n)] \to 0.$$ Then, by invoking Proposition \ref{prop:main} the result in \eqref{eq:Deltaregular} follows. 

\item {\it $W$ is is not $H$-regular} (Theorem \ref{thm:irregular}): This is proved in Section \ref{sec:thm2proof}. There are three cases: 

\begin{itemize}

\item $n\rho_n^{m_1(H)} \ll 1$: This is similar to the regular regime. A variance computation shows that 
$$\Var[\Delta_2(H,G_n)] \ll \Var[\Delta(H, G_n)] \to 0.$$ Then, by invoking Proposition \ref{prop:main} the result in \eqref{eq:Deltairregularb} follows.

\item $n\rho_n^{m_1(H)} \gg 1$:  In this case, a variance computation shows that 
$$\Var[\Delta_1(H,G_n)] \ll \Var[\Delta(H, G_n)] \to 0.$$ Hence, only $\Delta_2(H,G_n)$ contributes to the fluctuations of $\Delta(H, G_n)$. In this case, $\Delta_2(H,G_n)$ is a non-degenerate $U$-statistic of order $|E(H)|$ (since $W$ is not $H$-regular). Hence, invoking standard results about the asymptotic normality of $U$-statistics the result in \eqref{eq:Deltairregularu} follows. 

\item  $n\rho_n^{m_1(H)} \rightarrow c \in (0, \infty)$: In this case, $\Delta_1(H,G_n)$ and $\Delta_2(H,G_n)$ both contribute to the fluctuations of $\Delta(H, G_n)$. In particular, by combining Proposition \ref{prop:main} (for $\Delta_1(H,G_n)$), the asymptotic normality of non-degenerate $U$-statistics (for $\Delta_2(H,G_n)$), and Lemma \ref{lm:conditionaljoint} gives,   \begin{equation*}
    \begin{pmatrix} 
    \frac{\Delta_1(H,G_n)}{\sqrt{\Var[\Delta_1(H, G_n)]}} \\ 
    \frac{\Delta_2(H,G_n)}{\sqrt{\Var[\Delta_2(H, G_n)]}}
    \end{pmatrix} \dto N \left( 
    \begin{pmatrix} 
    0 \\ 
    0
    \end{pmatrix},   \begin{pmatrix} 
    1 & 0 \\ 
    0 & 1  
    \end{pmatrix} \right) . 
    \end{equation*}  
Then by a direct computation of the limit of $\frac{\Var[\Delta_1(H,G_n)]}{\Var[\Delta(H, G_n)]}$, the result in \eqref{eq:ZGnHirregular} follows.
\end{itemize}

\end{itemize}

\subsection{Prior Work}
\label{sec:Wsparse}

In this section, we summarize prior work on the fluctuations of subgraph counts in inhomogeneous random graph models. For clarity, we organize the discussion into three subsections, focusing separately on the (homogeneous) Erd\H{o}s-R\'{e}nyi model, the sparse graphon model, and the dense graphon model.

\subsubsection{Erd\H{o}s-R\'enyi Model} 
In the Erd\H{o}s-R\'enyi model $G(n, \rho_n)$ (which corresponds to $W \equiv 1$), it is well-known that $Z_n(H, G_n)$ is asymptotically normal if and only if $n\rho_n^{m(H)} \gg 1$ and $n^2(1-\rho_n) \gg 1$. This result was established independently by \cite{rucinski1988small} (using the method of moments), \cite{barbour1989central} (using Stein's method), \cite{nowicki1989asymptotic} (using $U$-statistics), and \cite{janson1990functional} (using martingales), among others (see \citet[Chapter 6]{JLR} for a comprehensive discussion). Note that in the sparse regime $\rho_n \ll 1$, the condition $n^2(1-\rho_n) \gg 1$ is always satisfied, ensuring asymptotic normality for all subgraph counts $X_n(H, G_n)$ whenever  $n\rho_n^{m(H)} \gg 1$. Corollary \ref{cor:ZHGn} shows that the same result hold for the sparse $W$-random model as well.

\subsubsection{Sparse Graphon Model} 
\label{sec:Wsparseresults}

Central limit theorem for subgraph counts in the sparse graphon model was first derived by \citet{bickel2011method}. The sparsity conditions in their result depends on whether $H$ is acyclic (that is, $H$ does not contain a cycle) or whether it contains a cycle. Specifically, they showed the following results. Throughout, suppose $G_n \sim G(n, \rho_n, W)$. 

\begin{itemize}

\item {\it $H$ is acyclic}: In this case, \citet[Theorem 1]{bickel2011method} shows that  $X_n(H, G_n)$ (appropriately standardized) is asymptotically normal, whenever $n \rho_n \gg 1$.\footnote{Technically, \citet{bickel2011method} considers the number of induced copies, but their proof can be easily adapted for $X_n(H, G_n)$.} Note that for any acyclic graph $m(H)<1$. Hence, there is a gap between the sparsity regime covered in \citet[Theorem 1]{bickel2011method} and the full sparsity regime ($n \rho_n^{m(H)} \gg 1$), when $H$ is acyclic.

\item {\it $H$ is not acyclic}: In this case, \citet[Theorem 1]{bickel2011method} shows that  $X_n(H, G_n)$ is asymptotically normal, whenever $n \rho_n^{\frac{|V(H)|}{2}} \gg 1$. Recalling \eqref{eq:mH} note that 
$$m(H) =\max_{F\subseteq H,\,|V(F)|\ge1}\frac{|E(F)|}{|V(F)|} \leq \max_{F\subseteq H,\,|V(F)|\ge1}\frac{ {|V(F)| \choose 2} }{|V(F)|} \leq \frac{|V(H)|-1}{2} < \frac{|V(H)|}{2}. $$
Therefore, there is again a gap between the sparsity regime in \citet[Theorem 1]{bickel2011method} and the full sparsity regime ($n \rho_n^{m(H)} \gg 1$), when $H$ contains a cycle. 

\end{itemize}
Later, \citet{zhang2022edgeworth} strengthened above the results by proving Edgeworth expansions and Berry-Essen bounds, under stronger sparsity conditions. Very recently, using the method of cumulants \citet{liu2026gaussian} derived normal approximation bounds in the Kolmogorov distance for $Z_n(H, G_n)$, when $H$ is strongly balanced. 
Notably, the rate of convergence in \citet[Corollary~7.8]{liu2026gaussian} depends on whether $n\rho_n^{m_1(H)} \ll 1$ or $n\rho_n^{m_1(H)} \gg 1$, which, interestingly, coincides with the threshold identified in Theorem~\ref{thm:irregular}, where the dominant contribution to the total variance shifts from $\Delta_1(H, G_n)$ and $\Delta_2(H, G_n)$. The results in \cite{liu2026gaussian}, however, do not consider the case where $H$ is not strongly balanced, or when $n\rho_n^{m_1(H)} \asymp 1$. In comparison, our results apply to all fixed subgraphs (not necessarily strongly balanced) and cover the entire sparsity regime, albeit without explicit rates of convergence. 

\subsubsection{Dense Graphon Model} 

Throughout this paper we have considered the sparse graphon model (where $\rho_n \ll 1$). In the dense regime (when $\rho_n \asymp 1$), the asymptotic normality of 
$X_n(H, G_n)$ for all graphs $H$ was established  by \citet{feray2020graphons} (see also \cite{invariant2022limit,delmas2021asymptotic,kaur2021higher,zhang2021berryesseen} for alternate proofs, extensions, and related results). 
However, the limiting normal distribution of $X_n(H, G_n)$ as derived in \cite{feray2020graphons}, can be degenerate depending on the structure of $W$. This degeneracy phenomenon was investigated by \citet{hladky2021limit} when $H = K_r$ is the $r$-clique (the complete graph on $r$ vertices), for general graphs $H$ by \citet{bhattacharya2021fluctuations}, and for joint distributions by \citet{graphonjoint}. Specifically, these results show that when $W$ is not $H$-regular (recall Definition \ref{regular}), then $X_n(H, G_n)$ is asymptotically Gaussian, with a normalization factor of $n^{|V(H)| - \frac{1}{2}}$. However, when $W$ is $H$-regular, then the normalization factor becomes $n^{|V(H)| - 1}$ and the limiting distribution of $X_n(H, G_n)$ has, in general, a Gaussian component and another independent (non-Gaussian) component which is a (possibly) infinite weighted sum of centered chi-squared random variables. This degeneracy phenomenon also appears in the subsequent work of \citet{chatterjee2024fluctuation} on the fluctuations of the largest eigenvalue. 
Recently, \citet{huang2024gaussian} established an invariance principle for $X_n(H, G_n)$ that encompasses higher-order degeneracies.

\section{Proof of Proposition \ref{prop:threshold} }
\label{sec:thresholdpf}

The proof of Proposition \ref{prop:threshold} is a standard application of the first and second moment methods. Specifically, it relies on the following lemma, which derives the asymptotic orders of the mean and variance of $X_n(H, G_n)$.

\begin{lem} Suppose Assumption \ref{assumption} holds. 
Then for $G_n \sim G(n, \rho_n, W)$, 
\begin{align}\label{eq:EHGn}
\E[X_n(H, G_n)]\asymp_{H, W} n^{|V(H)|}\rho_n^{|E(H)|}. 
\end{align} 
Moreover, for $\rho_n \ll 1$, 
\begin{align}\label{eq:VarGn}
\Var[X_n(H, G_n)] \asymp_{H, W} 
\left\{
\begin{array}{cc}
\displaystyle{\max_{F\subseteq H:|V(F)|\ge 1}} \frac{n^{2|V(H)|}\rho_n^{2|E(H)|}}{n^{|V(F)|}\rho_n^{|E(F)|}}  &   \text{ if } W \text{ is not } H\text{-regular} , \\
\displaystyle{\max_{F\subseteq H:|E(F)|\ge 1}} \frac{n^{2|V(H)|}\rho_n^{2|E(H)|}}{n^{|V(F)|}\rho_n^{|E(F)|}}  &    \text{ if } W \text{ is } H\text{-regular}  . 
\end{array}
\right.  
\end{align}  
\label{lem:meanvar}
\end{lem}

The proof of Lemma \ref{lem:meanvar} is given in Section \ref{sec:meanvarpf}. Here, using this result, we complete the proof of Proposition \ref{prop:threshold}. First, suppose $n \rho_n^{m(H)} \ll 1$. Then there are two cases: 

\begin{itemize} 

\item {\it $H$ is balanced}: In this case, $m(H) = \frac{|E(H)|}{|V(H)|}$ and $n \rho_n^{m(H)} \ll 1$ implies $n^{|V(H)|}\rho_n^{|E(H)|} \ll 1$, that is, $\E[X_n(H,G_n)]\to 0$ (recall \eqref{eq:EHGn}). Therefore, by Markov's inequality, 
\begin{align*}
\P(X_n(H,G_n)> 0) & =\P(X_n(H,G_n)\ge 1)  \le \E[X_n(H,G_n)] \rightarrow 0 . 
\end{align*} 

\item {\it $H$ is unbalanced}: Let $F\subseteq H$ be a subgraph of $H$ such that $m(H)=\frac{|E(F)|}{|V(F)|}$. Then $n \rho_n^{m(H)} \ll 1$ implies $n^{|V(F)|}\rho_n^{|E(F)|}\to 0$, that is, $\E[X_n(F,G_n)] \to 0$. This implies, 
$$\P(X_n(H,G_n)>0) \le \P(X_n(F,G_n)>0) \le \E[X_n(F,G_n)] \to 0.$$

\end{itemize}
The above two cases combined establish the 0-statement in Proposition \ref{prop:threshold}. 

Next, suppose $n \rho_n^{m(H)} \gg 1$. Then by Lemma \ref{lem:meanvar}, 
\begin{align}\label{eq:H1}
     \frac{\mathrm{Var}[X_n(H,G_n)]}{(\E[X_n(H,G_n)])^2} \lesssim_{H, W} \max_{F\subseteq H,\,|V(F)|\geq 1} \frac{1}{n^{|V(F)|} \rho_n^{|E(F)|}} . 
\end{align} 
Note that the condition $n \rho_n^{m(H)} \gg 1$ implies $n^{|V(F)|}\rho_n^{|E(F)|} \gg 1$, for  all $F\subseteq H$ such that $|V(F)| \geq 1$. Hence, the RHS of \eqref{eq:H1} tends to $0$ whenever $n \rho_n^{m(H)} \gg 1$. This means, 
$$\frac{X_n(H,G_n)}{\E[X_n(H,G_n)]} \Pto 1 \quad \implies \quad X_n(H,G_n) \Pto \infty,$$
since $\E[X_n(H,G_n)] \rightarrow \infty$, whenever $n \rho_n^{m(H)} \gg 1$. This completes the proof of Proposition \ref{prop:threshold}. \hfill $\Box$

\subsection{Proof of Lemma \ref{lem:meanvar}}
\label{sec:meanvarpf}

The result in \eqref{eq:EHGn} is an immediate consequence of \eqref{eq:EXHW} and the condition $t(H, W) > 0$ from Assumption \ref{assumption}.  

We now proceed with the proof of \eqref{eq:VarGn}. To this end, let $\mathcal{H}_n=\{H_1,H_2,\ldots,H_M\}$ be the collection of all copies of $H$ in the complete graph $K_n$. Here, $M=X_n(H, K_n)$ is the number of copies of $H$ in $K_n$. Now, define $I_s:=\bm{1}\{H_s\subseteq G_n\}$, for $1\le s \le M$. Then $\mathbb{E}[I_s]=\rho_n^{|E(H)|}t(H,W)$, for all $1\le s \le M.$ Furthermore,  
\begin{align}
    & \Var[X_n(H, G_n)] \nonumber \\ 
    & =\Var\left[\sum_{s=1}^M I_s \right] \nonumber \\
    &=\sum\limits_{\substack{H_s,H_t \in\mathcal{H}_n\\|V(H_s)\cap V(H_t)|\ge1}}\Cov[I_s,I_t]  \nonumber\\ 
    & =\sum_{F\subseteq H:|V(F)|\ge1}\sum_{\substack{H_s,H_t \in\mathcal{H}_n \\ H_s \cap H_t \simeq F}}\Cov[I_s,I_t] \nonumber\\
    &=\sum_{F\subseteq H:|V(F)|\ge1}\sum_{\substack{H_s,H_t \in\mathcal{H}_n \\ H_s \cap H_t \simeq F}}[\rho_n^{2|E(H)|-|E(F)|}t(H_s\cup H_t,W)-(\rho_n^{|E(H)|}t(H,W))^2] , 
    \label{eq:var-1}
\end{align} 
where $H_s\cup H_t=(V(H_s)\cup V(H_t),E(H_s)\cup E(H_t)).$ The second last line follows since for any two copies $H_s,H_t\in\mathcal{H}_n,$ the graph $H_s\cap H_t=(V(H_s)\cap V(H_t),E(H_s)\cap E(H_t))$ is isomorphic to some subgraph $F\subseteq H.$\footnote{For two graphs $G_1$ and $G_2,$ the notation $G_1\simeq G_2$ will mean that $G_1$ and $G_2$ are isomorphic.}

To show the upper bound and lower bound statements in \eqref{eq:VarGn}, we need the following lemma, which essentially states that for each subgraph $F\subseteq H,$ one can find a suitable number of copies $H_s,H_t$ of $H$ such that $H_s\cap H_t\simeq F$ and $H_s\cup H_t$ has positive homomorphism density $t(H_s\cup H_t,W)>0$ with respect to $W.$ The lemma is proved in Appendix \ref{sec:variancepf}.

\begin{lem}
Let $H$ be a subgraph with $|E(H)|\ge 1$ and $t(H,W)>0.$ Let $F\subseteq H$ be a subgraph of $H$ with $|V(F)|\ge1.$ Then there exist copies $H_s,H_t\in\mathcal{H}_n$ of $H$ in $K_n$ such that $H_s\cap H_t\simeq F$ and $t(H_s\cup H_t,W)\ge t(H,W)^2 > 0.$ Moreover, the number of such pairs $H_s,H_t$ is $\asymp_H(n^{2|V(H)|-|V(F)|}).$
\label{lem:positivity}
\end{lem}

We begin by decomposing \eqref{eq:var-1} according to the following 3 parts:  
\begin{align}
    &\Var[X_n(H,G_n)]\nonumber\\
    =&\,\sum_{F\subseteq H:|V(F)|=1}\sum_{\substack{ H_s,H_t \in\mathcal{H}_n \\ H_s \cap H_t \simeq F}}[\rho_n^{2|E(H)|-|E(F)|}t(H_s\cup H_t,W)-(\rho_n^{|E(H)|}t(H,W))^2]\nonumber\\
    &\,+\sum_{\substack{F\subseteq H:|V(F)|\ge2,\\|E(F)|=0}}\sum_{\substack{H_s,H_t \in\mathcal{H}_n \\ H_s \cap H_t \simeq F}}[\rho_n^{2|E(H)|-|E(F)|}t(H_s\cup H_t,W)-(\rho_n^{|E(H)|}t(H,W))^2]\nonumber\\
    &\,+\sum_{\substack{F\subseteq H:|V(F)|\ge2,\\|E(F)|\ge1}}\sum_{\substack{H_s,H_t \in\mathcal{H}_n \\ H_s \cap H_t \simeq F}}[\rho_n^{2|E(H)|-|E(F)|}t(H_s\cup H_t,W)-(\rho_n^{|E(H)|}t(H,W))^2]\nonumber\\
    =:&\,T_1+T_2+T_3.
    \label{eq:var-decomp}
\end{align}
We begin by considering $T_1$. Since in this case $|E(F)|=0,$ we may factor out $\rho_n^{2|E(H)|}$ to obtain
\begin{equation}
    T_1=\rho_n^{2|E(H)|}\sum_{\substack{H_s,H_t\in\mathcal{H}_n\\H_s\cap H_t\simeq K_1}}[t(H_s\cup H_t,W)-t(H,W)^2] , 
    \label{eq:var-3.2}
\end{equation}
where $K_1$ denotes a single isolated vertex. Now, to identify the structure of $H_s\cup H_t,$   we need the following definition: 
\begin{defn}
Suppose $H = (V(H), E(H))$ with vertices labeled $V(H)  = \{1, 2, \ldots, |V(H)|\}$. Then for $1 \leq a, b \leq |V(H)|$, the $(a, b)$-vertex join of two copies of $H,$ denoted by $H\bigoplus_{a,b} H$, is the graph obtained by identifying the $a$-th vertex of the first copy $H$ with the $b$-th vertex of the second copy $H$ (see Figure \ref{fig:v_join}).
\label{defn:v_join}
\end{defn}

\begin{figure}
    \centering
    \includegraphics[scale=0.65]{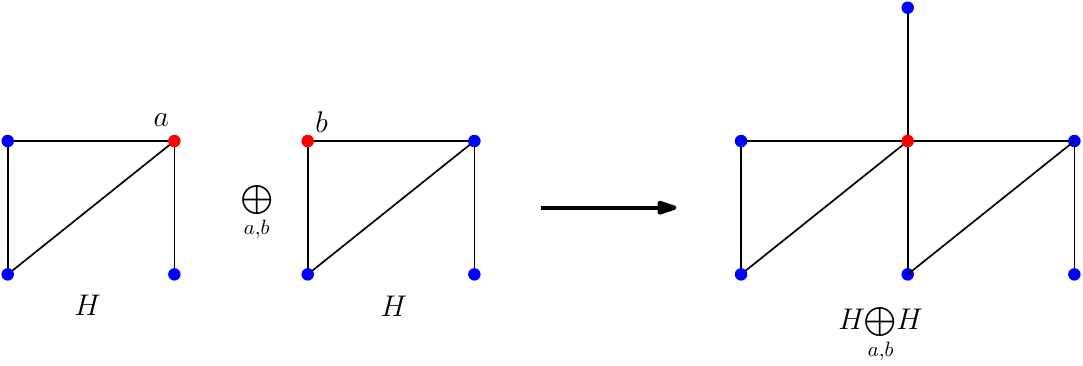}
    \caption{The $(a, b)$-{\it vertex join} of the 2 copies of $H$. }
    \label{fig:v_join}
\end{figure}

Note that if $H_s,H_t \in\mathcal{H}_n$ such that  $H_s\cap H_t\simeq K_1,$ then $H_s\cup H_t$ is isomorphic to $H\bigoplus_{a,b}H$ for some choice of $a, b \in V(H).$ Also, observe that for each fixed pair $(a,b)$, the number of ordered pairs $(H_s,H_t)$ such that $H_s\cap H_t\simeq K_1$ is $\asymp_H n^{2|V(H)|-1} $. Hence, 
\begin{align*}
    T_1&=\rho_n^{2|E(H)|}\sum_{\substack{H_s,H_t\in\mathcal{H}_n\\H_s\cap H_t\simeq K_1}}[t(H_s\cup H_t,W)-t(H,W)^2]\\
    &\asymp_{H, W} n^{2|V(H)|-1} \rho_n^{2|E(H)|} \left[\sum_{1\le a,b\le |V(H)|}t\left(H\bigoplus_{a,b} H,W\right)-|V(H)|^2t(H,W)^2\right].
\end{align*}
By \citet[Lemma 5.4]{bhattacharya2021fluctuations} the bracketed term in the above expression is zero if and only if $W$ is $H$-regular. Thus,
\begin{equation}
    T_1\asymp_{H,W}\begin{cases}n^{2|V(H)|-1}\rho_n^{2|E(H)|}&\text{if $W$ is not $H$-regular,}\\0&\text{if $W$ is $H$-regular.}\end{cases}
    \label{eq:T1}
\end{equation}
Next, we consider $T_2$. Since $|E(F)|=0,$ we again factor out $\rho_n^{2|E(H)|}$ to get, 
$$T_2=\sum_{\substack{F\subseteq H:|V(F)|\ge2,\\|E(F)|=0}}\rho_n^{2|E(H)|}\sum_{\substack{H_s,H_t \in\mathcal{H}_n \\ H_s \cap H_t \simeq F}}[t(H_s\cup H_t,W)-t(H,W)^2].$$
For a fixed subgraph $F\subseteq H,$ the number of ordered pairs $(H_s,H_t)$ with $H_s\cap H_t \simeq F$ is $\asymp_H n^{2|V(H)|-|V(F)|}.$ Hence, 
\begin{equation}
    |T_2|\lesssim_{H,W}\sum_{\substack{F\subseteq H:|V(F)|\ge2,\\|E(F)|=0}} n^{2|V(H)|-|V(F)|}\rho_n^{2|E(H)|} \lesssim_{H, W} n^{2|V(H)|-2}\rho_n^{2|E(H)|}.
    \label{eq:T2}
\end{equation}
Finally, we consider $T_3$: 
\begin{equation}
    T_3=\sum_{\substack{F\subseteq H:|V(F)|\ge2,\\|E(F)|\ge1}}\sum_{\substack{H_s,H_t \in\mathcal{H}_n \\ H_s \cap H_t \simeq F}}[\rho_n^{2|E(H)|-|E(F)|}t(H_s\cup H_t,W)-(\rho_n^{|E(H)|}t(H,W))^2].
    \label{eq:T3-1}
\end{equation} 
Recall that for each fixed subgraph $F\subseteq H,$ the number of ordered pairs $(H_s,H_t)$ with $H_s\cap H_t \simeq F$ is $\asymp_H n^{2|V(H)|-|V(F)|}.$ Therefore, if dropping the second term inside the bracket in the above expression and upper bound each of the homomorphism densities $t(H_s\cup H_t,W)$ by 1 gives, 
\begin{equation}
    T_3\lesssim_{H,W}\sum_{\substack{F\subseteq H:|V(F)|\ge2,\\|E(F)|\ge1}}n^{2|V(H)|-|V(F)|}\rho_n^{2|E(H)|-|E(F)|}.
    \label{eq:T3-upper}
\end{equation}
For the matching lower bound, fix $F\subseteq H$ with $|E(F)|\ge1.$ By Lemma \ref{lem:positivity}, for each subgraph $F\subseteq H,$ we can find $\asymp_H n^{2|V(H)|-|V(F)|}$ pairs $(H_s,H_t)$ with $H_s\cap H_t\simeq F$ and $t(H_s\cup H_t,W)>0.$ Restricting the inner sum in \eqref{eq:T3-1} to these pairs, the first term inside the bracket in \eqref{eq:T3-1} is strictly positive. Moreover, since $\rho_n\ll1$ and $|E(F)|\ge1,$
$$(\rho_n^{|E(H)|}t(H,W))^2 \ll \rho_n^{2|E(H)|-|E(F)|}.$$
Consequently,
\begin{equation}
    T_3\gtrsim_{H,W}\sum_{\substack{F\subseteq H:|V(F)|\ge2,\\|E(F)|\ge1}}n^{2|V(H)|-|V(F)|}\rho_n^{2|E(H)|-|E(F)|}.
    \label{eq:T3-lower}
\end{equation}
Combining \eqref{eq:T3-upper} and \eqref{eq:T3-lower}, we conclude
\begin{align}
    T_3 & \asymp_{H,W}\sum_{\substack{F\subseteq H:|V(F)|\ge2,\\|E(F)|\ge1}}n^{2|V(H)|-|V(F)|}\rho_n^{2|E(H)|-|E(F)|} \nonumber \\ 
    & \asymp_{H,W}\max_{F\subseteq H:|E(F)|\ge1}\frac{n^{2|V(H)|}\rho_n^{2|E(H)|}}{n^{ |V(F)|}\rho_n^{|E(F)|}} .
    \label{eq:T3}
\end{align}
Note that, recalling \eqref{eq:T2}, 
\begin{align*} 
\frac{ |T_2| }{T_3} \asymp_{H,W} \min_{F\subseteq H:|E(F)|\ge1} n^{|V(F)|-2}\rho_n^{|E(F)|} \lesssim_{H} \rho_n \ll 1 . 
 \end{align*} 
Hence, $|T_2| \ll T_3$ and combining \eqref{eq:T1} and \eqref{eq:T3} with \eqref{eq:var-decomp} gives the following: 
\begin{itemize}

\item If $W$ is not $H$-regular, then 
\begin{align*}
    \Var[X_n(H,G_n)]& \asymp_{H,W} n^{2|V(H)|-1}\rho_n^{2|E(H)|} + \max_{F\subseteq H:|E(F)|\ge1}\frac{n^{2|V(H)|}\rho_n^{2|E(H)|}}{n^{ |V(F)|}\rho_n^{ |E(F)|}} \\
    &\asymp_{H,W}\max_{F\subseteq H:|V(F)|\ge1}\frac{n^{2|V(H)|}\rho_n^{2|E(H)|}}{n^{|V(F)|}\rho_n^{2|E(F)|}}.
\end{align*} 

\item If  $W$ is $H$-regular,
\begin{align*}
    \Var[X_n(H,G_n)] & \asymp_{H,W}\max_{F\subseteq H:|E(F)|\ge1}\frac{n^{2|V(H)|}\rho_n^{2|E(H)|}}{n^{|V(F)|}\rho_n^{ |E(F)|}}.
\end{align*}
\end{itemize}
This completes the proof of \eqref{eq:VarGn}.\hfill $\Box$

\section{Proof of Proposition \ref{prop:main}} 
\label{sec:conditionalpf}

The proof of Proposition \ref{prop:main} will use Stein's method based on  dependency graphs. We begin by defining the notion of a conditional dependency graph. 

\begin{defn}[Conditional dependency graph]
    Let $\{X_v\}_{v\in V}$ be a family of random variables (defined on some common probability space) indexed by a finite set $V$ and $\mathcal{F}$ be a $\sigma$-algebra. Then a graph $\mathcal G$ with vertex set $V$ is said to be a conditional dependency graph for the collection $\{X_v\}$ given $\mathcal{F}$ if the following holds: For any two subsets of vertices $A, B \subseteq V$ such that there is no edge in $\mathcal{G}$ from any vertex in $A$ to any vertex in $B,$ then the colelctions $\{X_v\}_{v\in A}$ and $\{X_v\}_{v\in B}$ are conditionally independent given $\mathcal{F}.$ 
\end{defn}

For a dependency graph $\mathcal G$ and $u \in V$, denote by $\bar{N}_\mathcal{G}(u)$ the set neighbors of $u$ in $\mathcal{G}$ and $u$ itself. Also,  denote $\bar{N}_\mathcal{G}(u,v):=\bar{N}_\mathcal{G}(u)\cup\bar{N}_\mathcal{G}(v).$ We will use the following version of Stein's method based on dependency graph.

\begin{prop}[{\cite[Theorem 6.33]{JLR}}]
\label{prop:Stein}
    Suppose $\{X_v\}_{v\in V}$ be a family of random variables with conditional dependency graph $\mathcal{G}$ given $\mathcal{F}.$  Assume $\E[X_v|\mathcal{F}]=0$ almost surely, for all $v\in V.$ Let $$W=\frac{1}{\sigma(\mathcal{F})}\sum_{v\in V}X_v,$$ 
    where $(\sigma(\mathcal{F}))^2=\Var[\sum_{v\in V}X_v\mid\mathcal{F}]$ is the conditional variance. Assume, for all $u,v\in V,$ 
    \begin{equation}
        \sum_{v\in V}\E[|X_v|\mid\mathcal{F}] \le R(\mathcal{F})\quad\text{and}\quad\sum_{w\in\bar{N}_\mathcal{G}(u,v)}\E[|X_w|\mid X_u,X_v,\mathcal{F}] \le Q(\mathcal{F}) , 
        \label{eq:lem.asm}
    \end{equation}
    almost surely. Then, for $Z \sim N(0, 1)$,  $$d_\mathrm{Wass}(W\mid\mathcal{F},Z)\lesssim \frac{R(\mathcal{F})(Q(\mathcal{F}))^2}{(\sigma(\mathcal{F}))^3}$$ 
    almost surely.
\end{prop}


\begin{remark} Technically, \citet[Theorem 6.33]{JLR} establishes the above result in the unconditional setting. The conditional result stated above follows by repeating the proof verbatim (with expectations replaced by conditional expectations). 
\end{remark}

With the above preparations we can now proceed with the proof of Proposition \ref{prop:main}. To begin with, let $\mathcal{H}_n=\{H_1,H_2,\ldots,H_M\}$ be the collection of copies of all $H$ in the complete graph $K_n$. For $1\le s \le M$, define $I_s:=\bm{1}\{H_j\subseteq G_n\}$ and 
$$X_s:=I_s-\mathbb{E}[I_s \mid \CF(\bm{U}_n) ] =I_s-\rho_n^{|E(H)|}\prod_{(a,b)\in E(H_s)}W(U_a,U_b).$$
Clearly, $\E[X_s \mid \mathcal F] = 0$, for all $1 \leq s \leq M$. Recalling the definition of $\Delta_1(H,G_n)$ from \eqref{eq:XHGn12}, note that 
$$\Delta_1(H,G_n)=\sum_{s=1}^{M}X_s.$$ 
Construct a conditional dependency graph $\mathcal G$ of the collection of random variables $\{X_s :H_s \in\mathcal{H}_n\}$ given $\sigma(\bm U_n)$ on the vertex set $\{1,2,\ldots,M\}$ as follows:  Connect the edge $(s,t)$, for $1 \leq s \ne t \leq M$ in $\mathcal{G}_n$ if and only if $|E(H_s)\cap E(H_t)|\ge 1$.  

With this dependency graph, we now bound the terms appearing in Proposition \ref{prop:Stein}. We begin the first term in \eqref{eq:lem.asm}. 

\begin{lem}\label{lm:Rn} For $X_s$ as defined above, 
\begin{align}\label{eq:Rn}
\sum_{s = 1}^M \E[|X_s| \mid \CF(\bm U_n)] \lesssim_{H,W} \E[X_n(H, G_n)] , 
\end{align} 
almost surely. 
\end{lem} 

\begin{proof} For $1 \leq s \leq M$, 
\begin{align}
    \E[|X_s|\mid \CF(\bm{U}_n) ] &=2 \E[I_s \mid \CF(\bm{U}_n) ] (1-\E[I_s \mid \CF(\bm{U}_n) ] ) \nonumber \\
    &=2\rho_n^{|E(H)|}\prod_{(a,b )\in E(H_s)}W(U_a,U_b)\left(1-\rho_n^{|E(H)|}\prod_{(a,b)\in E(H_s)}W(U_a,U_b)\right) \nonumber \\
    & \lesssim_{H,W}\rho_n^{|E(H)|} , 
    \label{eq:Xs}
    \end{align}
almost surely. Since $M= |\mathcal H_n|=\binom{n}{|V(H)|}\frac{|V(H)|!}{|\Aut(H)|}\lesssim_{H,W} n^{|V(H)|}$,
\begin{align}
\sum_{s=1}^M\E[|X_s|\mid \CF(\bm{U}_n)]\lesssim_{H,W} n^{|V(H)|}\rho_n^{|E(H)|}\asymp_{H,W}\E[X_n(H,G_n)] , 
\label{eq:RnHW}
\end{align}
where the last step uses \eqref{eq:EHGn}. Thus, Lemma \ref{lm:Rn} follows.  
\end{proof}

Next, we consider the second term in \eqref{eq:lem.asm}. 

\begin{lem}\label{lm:Qn} For all $1 \leq s < t \leq M$
\begin{align}\label{eq:Qn}
\sum_{w\in\bar{N}_\mathcal{G}(s,t)}\E[|X_w|\mid X_s,X_t,\CF(\bm U_n)] \lesssim_{H, W} \frac{\E[X_n(H, G_n)]}{\Phi_H} ,
\end{align}
almost surely, where
$$\Phi_H:=\min\{\E[X_n(F,G_n)] : F\subseteq H,|E(F)|\ge 1\}\asymp\min_{F\subseteq H,\,|E(F)|\ge 1}n^{|V(F)|}\rho_n^{|E(F)|}.$$ 
\end{lem}

\begin{proof} Fix $H_s,H_t \in\mathcal{H}_n$, for $1 \leq s < t \leq M$. 
For every $H_w \in\mathcal{H}_n,$ define $H_{w,\{s,t\} }= H_w \cap (H_s\cup H_t).$ Now, observe that $H_{w, \{s, t\}}$ is isomorphic to some subgraph $F\subseteq H$, since $H_w$ is a copy of $H.$ Then, almost surely, 
\begin{align}\label{eq:Xwst}
\E[|X_w|\mid X_s,X_t, \CF(\bm U_n)] & \le\E[I_w\mid I_s,I_t,\CF(\bm U_n)]+ \rho_n^{|E(H)|}\prod_{(a,b)\in E(H_w)}W(U_a,U_b) \nonumber \\ 
&=\rho_n^{|E(H)|-|E(F)|}\prod_{(a,b)\in E(H_w)\setminus E(H_{w,\{s,t\}})}W(U_a,U_b)\nonumber\\
&\quad\quad\quad\quad\quad\quad\quad\quad\quad\quad+\rho_n^{|E(H)|}\prod_{(a,b)\in E(H_w)}W(U_a,U_b)\nonumber\\
&\le \rho_n^{|E(H)|-|E(F)|} + \rho_n^{|E(H)|} \nonumber\\
&\lesssim_{H,W}\rho_n^{|E(H)|-|E(F)|}, 
\end{align} 
since $\rho_n \ll 1$.  If we fix a subgraph $F\subseteq H,$ then the number of choices of $w \in\{1,2,\ldots,M\}$ such that $H_{w,\{s,t\}}$ is isomorphic to $F$ is $O_H(n^{|V(H)|-|V(F)|})$. Furthermore, note that $w\in\bar{N}_{\mathcal{G}}(s,t)$ if and only if $|E(H_{w,\{s,t\}})|>0.$ Hence, using \eqref{eq:Xwst}, the LHS of \eqref{eq:Qn} can be bounded as: 
\begin{align*}
    \sum_{w\in\bar{N}_\mathcal{G}(s,t)}\E[|X_w|\mid X_s,X_t,\CF(\bm U_n)]& \lesssim_{H, W} \sum_{F\subseteq H:|E(F)|\ge1}n^{|V(H)|-|V(F)|}\rho_n^{|E(H)|-|E(F)|} \\
    &\asymp_{H, W} \max_{F\subseteq H:|E(F)|\ge 1}n^{|V(H)|-|V(F)|}\rho_n^{|E(H)|-|E(F)|}\nonumber\\
    &\asymp_{H,W}\frac{\E[X_n(H,G_n)]}{\Phi_H}, 
\end{align*}
almost surely due to \eqref{eq:EHGn}. Thus Lemma \ref{lm:Rn} follows. 
\end{proof}

Finally, to apply Proposition \ref{prop:Stein} we need a lower bound on the conditional variance $\sigma(\CF(\bm U_n))$ where
\begin{align}\label{eq:varianceproposition}
(\sigma(\CF(\bm U_n)))^2 = \Var[\Delta_1(H,G_n)|\CF(\bm U_n)] = \sum\limits_{\substack{H_s,H_t\in\mathcal{H}_n\\|E(H_s)\cap E(H_t)|\ge1}}\Cov[I_s,I_t\mid \CF(\bm{U}_n) ], 
\end{align}
since the conditional covariance is $0$ when $|E(H_s)\cap E(H_t)|=0.$

\begin{lem}\label{lm:conditionalvariance} 
For $\sigma(\CF(\bm U_n))$ as defined before, we have
\begin{align}\label{eq:Sn}
(\sigma(\CF(\bm U_n)))^2 \gtrsim_{H, W}\frac{(\E[X_n(H, G_n)])^2}{\Phi_H} , 
\end{align} 
almost surely. 
\end{lem}

\begin{proof}
Note that if $|E(H_s)\cap E(H_t)|\ge1,$ then
\begin{align*}
    \Cov[I_s,I_t\mid \CF(\bm{U}_n) ]&=\rho_n^{|E(H_s)\cup E(H_t)|}\prod_{(a,b)\in E(H_s)\cup E(H_t)}W(U_{a}, U_{b}) - \rho_n^{2|E(H)|} A_n(s, t)
\end{align*}
where 
$$A_n(s, t) := \prod_{(a,b)\in E(H_s)}W(U_{a}, U_{b})\prod_{(c, d)\in E(H_t)}W(U_{c},U_{d}).$$ 
Note that $H_s\cap H_t=(V(H_s)\cap V(H_t),E(H_s)\cap E(H_t))$ is isomorphic to a subgraph $F$ of $H.$ This means, from \eqref{eq:varianceproposition}, 
\begin{align}\label{eq:sigmavariance}
    &(\sigma(\CF(\bm U_n)))^2=\sum\limits_{\substack{H_s,H_t\in\mathcal{H}_n\\|E(H_s)\cap E(H_t)|\ge1}}\Cov[I_s,I_t\mid \sigma(\bm{U}_n) ] \nonumber \\
    =&\sum\limits_{\substack{F\subseteq H\\|E(F)|\ge1}}  \left\{ \rho_n^{2|E(H)|-|E(F)|}\sum\limits_{\substack{H_s,H_t\in\mathcal{H}_n\\H_s\cap H_t\simeq F}}\prod_{(a,b)\in E(H_s)\cup E(H_t)}W(U_{a}, U_{b}) - \rho_n^{2|E(H)|} \sum\limits_{\substack{H_s,H_t\in\mathcal{H}_n\\ H_s \cap H_t \simeq F}} A_n(s, t) \right\} \nonumber \\ 
    =&\sum\limits_{\substack{F\subseteq H\\|E(F)|\ge1}} \left\{ \rho_n^{2|E(H)|-|E(F)|}Z_n(F) - \rho_n^{2|E(H)|} A_n(F) \right\} , 
\end{align}
almost surely, where
\begin{equation}
    Z_n(F):=\sum\limits_{\substack{H_s,H_t\in\mathcal{H}_n\\H_s\cap H_t\simeq F}}\prod_{(a,b)\in E(H_s)\cup E(H_t)}W(U_{a}, U_{b}) \text{ and } A_n(F)  := \sum\limits_{\substack{H_s,H_t\in\mathcal{H}_n\\ H_s \cap H_t \simeq F}} A_n(s, t) .
    \label{eq:ZnF}
\end{equation} 
Observe that $Z_n(F)$ and $A_n(F)$ are scaled $U$-statistics with kernels of order $2|V(H)|-|V(F)|.$ Hence, from the strong law for $U$-statistics (see \citet[Theorem 6.3.1]{martingale2024convergence}) it follows that $\frac{Z_n(F)}{\E[Z_n(F)]} \to1$ and $\frac{A_n(F)}{\E[A_n(F)]} \to1$ almost surely. Thus,
\begin{align}
    \frac{ \sigma(\CF(\bm{U}_n))^2}{\Var[\Delta_1(H,G_n)]} & =\frac{ \sigma(\CF(\bm{U}_n))^2}{\mathbb{E}[\Var[\Delta_1(H,G_n)\mid \mathcal{F}(\bm{U}_n) ]]} \tag*{(since $\E[\Delta_1(H,G_n))|\mathcal{F}(\bm{U}_n)]=0$)} \nonumber \\ 
    & = \frac{\sum_{F\subseteq H:|E(F)|\ge 1}\left\{ \rho_n^{2|E(H)|-|E(F)|}Z_n(F) - \rho_n^{2|E(H)|} A_n(F) \right\}}{\sum_{F\subseteq H:|E(F)|\ge 1} \left\{ \rho_n^{2|E(H)|-|E(F)|}\E[Z_n(F)] - \rho_n^{2|E(H)|} \E[A_n(F)] \right\} } \nonumber \\ 
    & \stackrel{a.s.} \rightarrow 1 , 
    \label{eq:ratio}
\end{align}
since there are finitely many terms in the sums in both numerator and denominator and for every fixed $F,$ $\frac{Z_n(F)}{\E[Z_n(F)]} \to1$ and $\frac{A_n(F)}{\E[A_n(F)]} \to1$, almost surely. Now,
\begin{equation}
    \Var[\Delta_1(H,G_n)]=\sum_{F\subseteq H:|E(F)|\ge 1} \left\{ \rho_n^{2|E(H)|-|E(F)|}\E[Z_n(F)] - \rho_n^{2|E(H)|} \E[A_n(F)] \right\}.
    \label{eq:var-delta1}
\end{equation}
Fix $F\subseteq H$ with $|E(F)|\ge1.$ Recall that by Lemma \ref{lem:positivity}, for each subgraph $F\subseteq H,$ we can find $\asymp_H(n^{2|V(H)|-|V(F)|})$ pairs $(H_s,H_t)$ with $H_s\cap H_t\simeq F$ and $t(H_s\cup H_t,W)>0.$ Thus, by definition of $A_n(F)$ and $Z_n(F),$ we have
$$\E[A_n(F)]\lesssim_{H,W}\E[Z_n(F)]\asymp_{H,W}n^{2|V(H)|-|V(F)|}.$$
Moreover, since $\rho_n\ll1$ and $|E(F)|\ge1,$
$$\rho_n^{2|E(H)|}\E[A_n(F)] \ll \rho_n^{2|E(H)|-|E(F)|}\E[Z_n(F)] . $$
Thus, the second term in the within the brackets in \eqref{eq:var-delta1} is of smaller order than the first term. Consequently,
\begin{align}
    \Var[\Delta_1(H,G_n)]&  \asymp_{H, W} \sum_{F\subseteq H:|E(F)|\ge 1}n^{2|V(H)|-|V(F)|}\rho_n^{2|E(H)| - |E(F)|}  \nonumber \\
    & \asymp_{H,W} \max_{F\subseteq H:|E(F)|\ge 1}\frac{n^{2|V(H)|}\rho_n^{2|E(H)|}}{n^{2|V(F)|}\rho_n^{2|E(F)|}}\asymp_{H,W} \frac{(\E[X(H, G_n)])^2}{\Phi_H} .
    \label{eq:sigmabound}
\end{align}
Combining \eqref{eq:ratio} and \eqref{eq:sigmabound}, the result in \eqref{eq:Sn} follows. 
\end{proof} 

Now, applying in the bounds \eqref{eq:Rn}, \eqref{eq:Qn}, and \eqref{eq:Sn}, in Proposition \ref{prop:Stein} gives 
$$d_\mathrm{Wass}\left((\sigma(\CF(\bm U_n)))^{-1}\Delta_1(H,G_n)|\CF(\bm U_n),Z\right)\lesssim_{H, W} \Phi_H^{-\frac{1}{2}}, $$
where $Z\sim N(0,1)$. Note that $\Phi_H \gg 1$ when $n\rho_n^{m(H)} \gg 1.$ This completes the proof of Proposition \ref{prop:main}. \hfill $\Box$

\section{Proof of Theorem \ref{thm:regular} }
\label{sec:thm1proof}

Recalling the definition of $\Delta_2(H,G_n)$ from \eqref{eq:XHGn12} note that 
\begin{align*}
    & \Delta_2(H,G_n) \nonumber \\ 
    &=\E[\Delta(H,G_n|\bm U_n)] - \E[\Delta(H, G_n)] \nonumber\\ 
    &=\sum_{1\le i_1<\ldots<i_{|V(H)|}\le n}\sum_{H'\in\mathscr{G}_H(\{i_1,\ldots,i_{|V(H)|}\})}\left(\prod_{(i_s,i_t)\in E(H')}\rho_n W(U_{i_s},U_{i_t})-\rho_n^{|E(H)|}t(H,W)\right) \nonumber \\
    &= \rho_n^{|E(H)|}\sum_{1\le i_1<\ldots<i_{|V(H)|}\le n} h(U_{i_1}, U_{i_2},\ldots, U_{i_{|V(H)|}}) , 
\end{align*} 
where 
$$h(x_{i_1}, x_{i_2},\ldots, x_{i_{|V(H)|}})  = \sum_{H'\in\mathscr{G}_H(\{i_1,\ldots,i_{|V(H)|}\})}\left(\prod_{(i_s,i_t)\in E(H')}W(x_{i_s},x_{i_t})-t(H,W)\right) . $$ 
Hence, defining 
\begin{align}\label{eq:THG}
T(H,G_n) := \frac{1}{\binom{n}{|V(H)|}}\sum_{1\le i_1<\ldots<i_{|V(H)|}\le n} h(U_{i_1}, U_{i_2},\ldots, U_{i_{|V(H)|}}), 
\end{align}
we get 
\begin{align}
\Delta_2(H,G_n) &=\binom{n}{|V(H)|}\rho_n^{|E(H)|}T(H,G_n). 
\label{eq:Delta2Ustatistics}
\end{align} 
Note that $T(H,G_n)$ is a centered $U$-statistics of order $|V(H)|$ with kernel $h.$ In the next lemma we show that the variance of $T(H, G_n)$ scales as $\frac{1}{n^2}$ when the graphon $W$ is $H$-regular. The proof is deferred to Appendix \ref{appendix:proofof_varlemma}. 

\begin{lem}\label{lm:HWU}
$T(H,G_n)$ is first-order degenerate if and only if the graphon $W$ is $H$-regular. Consequently, when $W$ is $H$-regular, $\Var[T(H,G_n)]\lesssim_{H,W}\frac{1}{n^2}.$ 
\end{lem}

Theorem \ref{thm:regular} assumes that $W$ is $H$-regular. Now, recalling \eqref{eq:Delta2Ustatistics}, and using Lemma \ref{lm:HWU},
$$\Var[\Delta_2(H,G_n)] \lesssim_{H, W} n^{2|V(H)|-2}\rho_n^{2|E(H)|}.$$ 
This means, recalling \eqref{eq:VarGn} in the $H$-regular case,
\begin{align} 
\frac{\Var[\Delta_2(H, G_n)]}{\Var[\Delta(H, G_n)]} \lesssim_{H, W} \min_{F\subseteq H:|E(F)|\ge 1} n^{|V(F)| -2 }\rho_n^{|E(F)|}\lesssim\rho_n  \rightarrow 0.
\end{align} 
This proves the first assertion in \eqref{eq:Deltaregular}. Next, using \eqref{eq:vardecomp},
$$\frac{\Var[\Delta_1(H,G_n)]}{\Var[\Delta(H, G_n)]}=1-\frac{\Var[\Delta_2(H,G_n)]}{\Var[\Delta(H, G_n)]}\to 1.$$
This means, from \eqref{eq:ratio}, $\frac{\Var[\Delta_1(H,G_n)|\CF(\bm U_n)]}{\Var[\Delta(H, G_n)]} \rightarrow 1$. Hence, applying Proposition \ref{prop:main} and Slutsky's theorem it follows that 
$$\frac{\Delta_1(H,G_n)}{\sqrt{\Var[\Delta(H, G_n)]}}\dto N(0,1),$$
unconditionally. This proves the second assertion in \eqref{eq:Deltaregular}. 
The result is \eqref{eq:ZGnHregular} is an immediate consequence of \eqref{eq:XHGn12} and \eqref{eq:Deltaregular}.    \hfill $\Box$

\section{Proof of Theorem \ref{thm:irregular}} \label{sec:thm2proof}

In this section, we prove Theorem \ref{thm:irregular}, which considers the case where $W$ is not $H$-regular. We begin by showing that $\Delta_2(H,G_n)$ (respectively, $\Delta_1(H,G_n)$) has negligible contribution to the total variance when  $n\rho_n^{m_1(H)} \ll 1$ (respectively, $n\rho_n^{m_1(H)} \gg 1$). 

\begin{prop}
\label{threshold2}
    Suppose $W$ is not $H$-regular. Then the following holds: 
    \begin{itemize}
    \item[$(1)$] If $n\rho_n^{m_1(H)} \ll 1,$ then $\Var[\Delta_2(H,G_n)]\ll\Var[\Delta(H, G_n)].$

\item[$(2)$] If $n\rho_n^{m_1(H)} \gg 1,$ then $\Var[\Delta_1(H,G_n)] \ll \Var[\Delta(H, G_n)] $. 
\end{itemize}
\end{prop}

The proof of Proposition \ref{threshold2} is given in Appendix \ref{sec:threshold2pf}. Using the above result, we now complete the proof of Theorem \ref{thm:irregular}. To begin with, suppose $n\rho_n^{m_1(H)} \ll 1$. The first assertion in \eqref{eq:Deltairregularb} follows from Proposition \ref{threshold2} (1). This means, using \eqref{eq:vardecomp}, $\frac{\Var[\Delta_1(H,G_n)]}{\Var[\Delta(H, G_n)]}=1-\frac{\Var[\Delta_2(H,G_n)]}{\Var[\Delta(H, G_n)]}\to 1$. Hence, from \eqref{eq:ratio}, $\frac{\Var[\Delta_1(H,G_n)|\CF(\bm U_n)]}{\Var[\Delta(H, G_n)]} \rightarrow 1$ almost surely. Hence, applying Proposition \ref{prop:main} and Slutsky's theorem it follows that 
$$\frac{\Delta_1(H,G_n)}{\sqrt{\Var[\Delta(H, G_n)]}}\dto N(0,1),$$
unconditionally. This proves the second assertion in \eqref{eq:Deltairregularb}. 

Next, suppose $n\rho_n^{m_1(H)} \gg 1$. The first assertion in \eqref{eq:Deltairregularu} follows from Proposition \ref{threshold2} (2). Also, since $W$ is not $H$-regular, by Lemma \ref{lm:HWU} $\Delta_2(H,G_n)$ is a scaled $U$-statistic of order $|V(H)|$ which is non-degenerate. Hence, by \citet[Chapter 3, Theorem 1]{lee2019ustatistics}, 
\begin{align}\label{eq:conditionalmeangaussian}
\frac{\Delta_2(H,G_n)}{\sqrt{\Var[\Delta_2(H,G_n)]}} \dto N(0, 1).
\end{align}
Then by Slutsky's lemma 
$$\frac{\Delta_2(H,G_n)}{\sqrt{\Var[\Delta(H,G_n)]}} \dto N(0, 1).$$
This proves the second assertion in \eqref{eq:Deltairregularu}.

Next, consider the case $n\rho_n^{m_1(H)}\to c,$ for some $c\in(0,\infty).$ We have, from \eqref{var-S2}, 
\begin{align*} 
\Var[\Delta_2(H,G_n)] & = n^{2|V(H)|-1}\rho_n^{2|E(H)|}\frac{|V(H)|^2}{(|V(H)|!)^2}\xi_1 + o(n^{2|V(H)|-1}\rho_n^{2|E(H)|}) . 
\end{align*} 
Also, we know from \eqref{eq:sigmavariance}, 
\begin{align*}
\Var[\Delta_1(H,G_n)] & = \E[ \Var[\Delta_1(H,G_n) | \mathcal{F}(\bm{U}_n)] ] \nonumber \\ 
& =\sum\limits_{\substack{F\subseteq H\\|E(F)|\ge1}} \left\{ \rho_n^{2|E(H)|-|E(F)|} \E [Z_n(F) ] - \rho_n^{2|E(H)|} \E [A_n(F) ] \right\} 
\end{align*}
where $Z_n(F)$ and $A_n(F)$ are defined in \eqref{eq:ZnF}. Note that under Assumption \ref{assumption} and Lemma \ref{lem:positivity},
$$\E[\rho_n^{2|E(H)|-|E(F)|}Z_n(F))] \asymp_{H,W} n^{2|V(H)|-|V(F)|}\rho_n^{2|E(H)|-|E(F)|}$$
and $\rho_n^{2|E(H)|} \E [A_n(F) ] = o(n^{2|V(H)|-|V(F)|}\rho_n^{2|E(H)|-|E(F)|}) $. 
Hence, 
\begin{align}
\hspace{-0.05in} \Var[\Delta_1(H,G_n)] & = \sum\limits_{\substack{F\subseteq H\\|E(F)|\ge1}} \left\{ \rho_n^{2|E(H)|-|E(F)|} \E [Z_n(F) ]  + o(n^{2|V(H)|-|V(F)|} \rho_n^{2|E(H)|-|E(F)|}) \right\}
\label{eq:var1-crit}
\end{align} 
When $n\rho_n^{m_1(H)}\to c\in(0,\infty),$ for any $F\subseteq H$ with $|V(F)|\ge2,$ we have $n^{|V(F)|-1}\rho_n^{|E(F)|}\gtrsim_{H,W}1$, or, equivalently, $n^{2|V(H)|-|V(F)|} \rho_n^{2|E(H)|-|E(F)|}\lesssim_{H,W}n^{2|V(H)|-1} \rho_n^{2|E(H)|}.$ Thus, continuing from \eqref{eq:var1-crit},
\begin{align*}
\Var[\Delta_1(H,G_n)] = \sum\limits_{\substack{F\subseteq H\\|E(F)|\ge1}} \rho_n^{2|E(H)|-|E(F)|} \E [Z_n(F) ]  + o(n^{2|V(H)|-1} \rho_n^{2|E(H)|}).
\end{align*} 
To find the dominating term in the above sum, define the following set: 
$$\mathscr{S}(H):=\left\{F\subseteq H:|E(F)|\ge1\text{ and }m_1(H)=\frac{|E(F)|}{|V(F)|-1}\right\}.$$
This is the collection of subgraphs $F$ of $H$ where the maximum in \eqref{eq:m1H} is attained. (Note that if $H$ is strictly strongly balanced $\mathscr{S}(H)=\{H\}$.)

\begin{lem} Suppose $n\rho_n^{m_1(H)}\to c\in(0,\infty).$ If $F\subseteq H$ is a subgraph with $|E(F)|\ge1$ and $F\notin\mathscr{S}(H)$, then $\E[\rho_n^{2|E(H)|-|E(F)|} Z_n(F)] \ll n^{2|V(H)|-1}\rho_n^{2|E(H)|} $. 
\label{lm:H}
\end{lem} 

\begin{proof}
Fix a subgraph $F\subseteq H$ with $|E(F)|\ge1$ and $F\not\in\mathscr{S}(H).$ Then $\frac{|E(F)|}{|V(F)|-1}<m_1(H).$ Since $n\rho_n^{m_1(H)}\to c\in(0,\infty),$ we must have $n^{|V(F)|-1}\rho_n^{|E(F)|}\gg1,$ or equivalently,
\begin{align}
    n^{2|V(H)|-|V(F)|}\rho_n^{2|E(H)|-|E(F)|}\ll n^{2|V(H)|-1}\rho_n^{2|E(H)|}.
    \label{eq:Z_nF_bound}
\end{align}
There are $O_H(n^{2|V(H)|-|V(F)|})$ copies $H_s,H_t\in\mathcal{H}_n$ of $H$ such that $H_s\cup H_t\simeq F.$ Thus, according to the definition of $Z_n(F)$ in \eqref{eq:ZnF}, we have
\begin{align*}
    \E[\rho_n^{2|E(H)|-|E(F)|} Z_n(F)]&\lesssim_{H,W}n^{2|V(H)|-|V(F)|}\rho_n^{2|E(H)|-|E(F)|} \\
    &\lesssim_{H,W}n^{2|V(H)|-|V(F)|}\rho_n^{2|E(H)|-|E(F)|}\ll n^{2|V(H)|-1}\rho_n^{2|E(H)|},
\end{align*}
where the last line follows from \eqref{eq:Z_nF_bound}. This completes the proof of the lemma.
\end{proof}

The above lemma implies, 
\begin{align}\label{eq:var-ratio} 
    & \frac{\Var[\Delta_2(H,G_n)]}{\Var[\Delta(H, G_n)]} \nonumber \\ 
    & = \frac{\Var[\Delta_2(H,G_n)]}{ \Var[\Delta_1(H, G_n)] + \Var[\Delta_2(H, G_n)] } \\ 
    & =\frac{n^{2|V(H)|-1}\rho_n^{2|E(H)|}\frac{|V(H)|^2}{(|V(H)|!)^2}\xi_1 + o(n^{2|V(H)|-1} \rho_n^{2|E(H)|}) }{n^{2|V(H)|-1}\rho_n^{2|E(H)|}\frac{|V(H)|^2}{(|V(H)|!)^2}\xi_1+\sum_{F\in\mathscr{S}(H)}\E[ \rho_n^{2|E(H)|-|E(F)|}Z_n(F) ]  + o(n^{2|V(H)|-1} \rho_n^{2|E(H)|}) }. \nonumber 
\end{align}
Hence, to compute the limit of \eqref{eq:var-ratio} we need to consider graphs which are in $\mathscr{S}(H)$. Fix graphs $H$ and $F\subseteq H,$ and define
$$\mathscr{T}(H,F):=\{R:H_s\cup H_t\simeq R,\text{ where }H_s,H_t \simeq H \text{ with }H_s\cap H_t\simeq F\}.$$
Also, define $\eta(H,F,R)$ to be the number of ways two copies of $H$ can be joined on the vertex set $\{1,2,\ldots,2|V(H)|-|V(F)|\}$ such that their intersection is isomorphic to $F$ and their union is isomorphic to $R.$ Formally, given $F\subseteq H$ and $R\in\mathscr{T}(H,F),$ we have
$$\eta(H,F,R)=|\{(H_1,H_2):H_1\simeq H_2\simeq H,H_1\cap H_2\simeq F,H_1\cup H_2\simeq R\}|.$$

\begin{lem} Suppose $n\rho_n^{m_1(H)}\to c\in(0,\infty).$ If $F \in\mathscr{S}(H)$, then 
$$\frac{\E[\rho_n^{2|E(H)|-|E(F)|} Z_n(F) ]}{n^{2|V(H)|-1} \rho_n^{2|E(H)|}} = \frac{1}{c^{|V(F)|-1}(2|V(H)|-|V(F)|)!}\sum_{R\in\mathscr{T}(H,F)}\eta(H,F,R)t(R,W) + o(1).$$
\label{lm:RW}
\end{lem} 

\begin{proof}
As $F\in\mathscr{S}(H),$ we have $n^{|V(F)|-1}\rho_n^{|E(F)|}\to c^{|V(F)|-1}.$ Then,
\begin{align*}
    & \frac{\E(\rho_n^{2|E(H)|-|E(F)|}Z_n(F))}{n^{2|V(H)|-1}\rho_n^{2|E(H)|}} \\ &=\frac{\rho_n^{2|E(H)|-|E(F)|}}{n^{2|V(H)|-1}\rho_n^{2|E(H)|}}\binom{n}{2|V(H)|-|V(F)|}\sum_{R\in\mathscr{T}(H,F)}\eta(H,F,R)t(R,W)\notag\\
    &\sim\frac{n^{2|V(H)|-|V(F)|}\rho_n^{2|E(H)|-|E(F)|}}{n^{2|V(H)|-1}\rho_n^{2|E(H)|}(2|V(H)|-|V(F)|)!}\sum_{R\in\mathscr{T}(H,F)}\eta(H,F,R)t(R,W)\nonumber\\
    &\sim\frac{1}{c^{|V(F)|-1}(2|V(H)|-|V(F)|)!}\sum_{R\in\mathscr{T}(H,F)}\eta(H,F,R)t(R,W).
\end{align*}  
This completes the proof of Lemma \ref{lm:RW}. 
\end{proof}

Applying Lemma \ref{lm:H} and Lemma \ref{lm:RW} to \eqref{eq:var-ratio} gives,  
\begin{align}\label{eq:varianceRW}
\kappa := 1- \frac{\frac{|V(H)|^2}{(|V(H)|!)^2}\xi_1}{\frac{|V(H)|^2}{(|V(H)|!)^2}\xi_1+\sum_{F\in\mathscr{S}(H)}\frac{1}{c^{|V(F)|-1}(2|V(H)|-|V(F)|)!}\sum_{R\in\mathscr{T}(H,F)}\eta(H,F,R)t(R,W)} . 
\end{align} 
Clearly, $\kappa \in (0, 1)$. This shows (recall \eqref{eq:vardecomp}), 
\begin{align}\label{eq:variancethreshold}
\frac{\Var[\Delta_1(H,G_n)]}{\Var[\Delta(H, G_n)]} \rightarrow \kappa \quad \text{ and } \quad \frac{\Var[\Delta_2(H,G_n)]}{\Var[\Delta(H, G_n)]} \rightarrow 1- \kappa . 
\end{align}
Now, using Proposition \ref{prop:main}, \eqref{eq:conditionalmeangaussian} (which also holds when $n\rho_n^{m_1(H)}\to c\in(0,\infty)$ by the same arguments), and Lemma \ref{lm:conditionaljoint}, we have 
    \begin{align*}
    \begin{pmatrix} 
    \frac{\Delta_1(H,G_n)}{\sqrt{\Var[\Delta_1(H, G_n)]}} \\ 
    \frac{\Delta_2(H,G_n)}{\sqrt{\Var[\Delta_2(H, G_n)]}}
    \end{pmatrix} \dto N \left( 
    \begin{pmatrix} 
    0 \\ 
    0
    \end{pmatrix},   \begin{pmatrix} 
    1 & 0 \\ 
    0 & 1  
    \end{pmatrix} \right) . 
    \end{align*}  
Then applying Slutsky's theorem with \eqref{eq:variancethreshold} proves the result in \eqref{eq:Deltairregulart}. 

\hfill $\Box$

\begin{remark} 
Note that when $H$ is strictly strongly balanced, that is, $\mathscr{S}(H)=\{H\}$, then expression of $\kappa$ in \eqref{eq:varianceRW} simplifies to: 
\begin{align*}
    \kappa&= 1- \frac{\frac{|V(H)|^2}{(|V(H)|!)^2}\xi_1}{\frac{|V(H)|^2}{(|V(H)|!)^2}\xi_1+\frac{1}{c^{|V(H)|-1}|V(H)|!}|\mathscr{G}_H|t(H,W)}\\
    &=\frac{t(H,W)/|\Aut(H)|}{t(H,W)/|\Aut(H)|+c^{|V(H)|-1}\xi_1/((|V(H)|-1)!)^2}, 
\end{align*}
recalling \eqref{eq:G_H}. Here, $|\mathscr{G}_H|$ is the number of copies of $H$ in the complete graph $K_{|V(H)|}$ on $|V(H)|$ vertices, and $\xi_1$ is defined in Appendix \ref{appendix:proofof_varlemma}. This shows that under $n\rho_n^{m_1(H)}\to c\in(0,\infty),$
$$\frac{\Var[\Delta_1(H,G_n)]}{\Var[\Delta(H, G_n)]}\to\kappa\quad\text{and}\quad\frac{\Var[\Delta_2(H,G_n)]}{\Var[\Delta(H, G_n)]}\to1-\kappa,$$
where $\kappa\in(0,1)$ is a constant depending on $c,H$ and $W.$ 
\end{remark} 

\section{Conclusions}

In this paper, we show that $X_n(H,G_n)$, with $G_n \sim G(n,\rho_n,W)$, is asymptotically normal whenever $n\rho_n^{m(H)} \gg 1$, that is, when the sparsity level is above the containment threshold for $H$. Another important regime which we do not consider in the paper is when $n\rho_n^{m(H)} \asymp 1$, that is, $\rho_n$ is at the containment threshold. For the classical Erd\H{o}s--R\'enyi model, it is known that at the threshold, $X_n(H,G_n)$ converges to a Poisson distribution when $H$ is strictly balanced \citep{bollobas1981threshold}. If $H$ is unbalanced, the problem reduces to a balanced subgraph via an appropriate normalization \cite[Section~4]{rucinski1990survey}. When $H$ is balanced but not strictly balanced, the limiting distribution is more intricate and depends on the structure and multiplicity of subgraphs maximizing $m(H)$, with no universal closed form \citep{bollobas1989subgraph,JLR}.

Recently, there has been growing interest in deriving analogous results for inhomogeneous random graph models. In this direction, \citet{coulson2016poisson} derived Poisson approximation results for $X_n(H,G_n)$, when $G_n$ is sampled from a stochastic block model and $H$ is strictly balanced (see also \cite{coulson2018compound}). Recently, \citet{liu2025normalpoisson} obtained an analogous Poisson approximation result in the random connection model. Extending these results to the sparse graphon model, as well as addressing the case where $H$ is not strictly balanced, when $\rho_n$ is at the containment threshold, are interesting directions for future research. 

\small

\subsection*{Acknowledgement} BBB and SC were supported in part by NSF CAREER grant DMS 2046393 and a Sloan Research Fellowship.

\bibliography{graphonbib,ref}
\bibliographystyle{abbrvnat}


\normalsize

\appendix

\section{Proofs of Technical Lemmas}  
\label{sec:lemproofs}

In this section we collect the proofs of some technical lemmas. We begin with the proof of   
Proposition \ref{prop:H} in Appendix \ref{sec:Hpf}. Lemma \ref{lem:positivity} is proved in Appendix \ref{sec:variancepf}. Then, in Appendix \ref{appendix:proofof_varlemma} we prove Lemma \ref{lm:HWU}. 
Finally, in Appendix \ref{sec:conditionaljointpf} we prove a result about establishing joint convergence from marginal convergence.

\subsection{Proof of Proposition \ref{prop:H}} 
\label{sec:Hpf} 

Since $|E(H)|\ge1,$ the maximum in the definition of $m(H)$ is not attained at a single isolated vertex. Thus,
$$m(H):=\max_{F\subseteq H,\,|V(F)|\ge1}\frac{|E(F)|}{|V(F)|}=\max_{F\subseteq H,\,|V(F)|\ge2}\frac{|E(F)|}{|V(F)|}<\max_{F\subseteq H,\,|V(F)|\ge2}\frac{|E(F)|}{|V(F)|-1}=m_1(H),$$
which proves the first part of the proposition.

Next, suppose $H$ is strongly balanced, that is, for all subgraphs $F = (V(F), E(F))$ of $H$, $\frac{|E(F)|}{|V(F)|-1} \leq \frac{|E(H)|}{|V(H)|-1}$. This means, for all $F$ such that $|V(F)| < |V(H)|$, 
\begin{align*}
\frac{|E(F)|}{|V(F)|} = \frac{|E(F)| (|V(F)|-1) }{(|V(F)|-1) |V(F)| } \leq  \frac{|E(H)| (|V(F)|-1) }{(|V(H)|-1) |V(F)| } & \leq \frac{|E(H)|}{|V(H)|} \left(\frac{|V(H)| (|V(F)|-1)}{|V(F)| (|V(H)|-1)} \right) \nonumber \\ 
& < \frac{|E(H)|}{|V(H)|} . 
\end{align*}
Moreover, if $|V(F)| = |V(H)|$, then for any subgraph $F$ which is not the whole graph $H$, $\frac{|E(F)|}{|V(F)|} < \frac{|E(H)|}{|V(H)|}$ since $|E(H)|\ge1$. This shows strongly balanced graphs are strictly balanced. \hfill $\Box$

\subsection{Proof of Lemma \ref{lem:positivity}}
\label{sec:variancepf}

Fix a subgraph $F\subseteq H$ with $|V(F)|\ge1.$ Consider two copies $H_s,\,H_t\in\mathcal{H}_n$ of $H$ in $K_n$ such that 
\begin{align*} 
V(H_s)&=\{i_1,\,i_2,\,\ldots,\,i_{|V(F)|},\,i_{|V(F)|+1},\,\ldots,\,i_{|V(H)|}\} \nonumber \\ 
V(H_t)&=\{i_1,\,i_2,\,\ldots,\,i_{|V(F)|},\,i_{|V(H)|+1},\,\ldots,\,i_{2|V(H)|-|V(F)|}\} , 
\end{align*} 
where the first $|V(F)|$ vertices are common. The edge sets of $H_s,\,H_t$ are determined by maps $\phi_s:V(H)\to V(H_s)$ and $\phi_t:V(H)\to V(H_t)$, such that $(a,\,b)\in E(H)$ if and only if $(\phi_s(a),\,\phi_s(b))\in E(H_s)$ and $(\phi_t(a),\,\phi_t(b))\in E(H_t)$, respectively. In particular, we consider the following maps
$$\phi_s(a)=i_a,\,1\le a\le |V(H)|,\quad\text{and}\quad\phi_t(a)=\begin{cases}i_a&\text{if }1\le a\le |V(F)|,\\i_{a+|V(H)|-|V(F)|}&\text{if }|V(F)|+1\le a\le |V(H)|.\end{cases}$$
Then it is easy to see that the first $|V(F)|$ vertices of $H_s$ and $H_t$ constitute a copy of $F,$ that is, $H_s\cap H_t\simeq F$, and the remaining part of the copy of $H$ is extended in a similar fashion for both $H_s$ and $H_t.$ Then
\begin{align*}
    t(H_s\cup H_t,\,W)&=\int_{[0,\,1]^{|V(H_s)\cup V(H_t)|}}\prod_{(a,\,b)\in E(H_s)\cup E(H_t)}W(U_a,\,U_b)\prod_{a\in V(H_s)\cup V(H_t)}\text{d}x_a\\
    &=\int_{[0,\,1]^{|V(F)|}}\frac{t^2_{V(F)}(x_1,\,x_2,\,\ldots,\,x_{|V(F)|};\,H,\,W)}{\prod_{(a,\,b)\in E(F)}W(x_a,\,x_b)}\prod_{a\in V(F)}\text{d}x_a , 
    \end{align*} 
    where 
$$t_{V(F)}(x_1,\,x_2,\,\ldots,\,x_{|V(F)|};\,H,\,W)=\E\left[\prod_{(a,\,b)\in E(H)}W(U_a,\,U_b)\mid U_{a_j}=x_j,\text{ for }1\le j\le |V(F)|\right],$$ 
is the $|V(F)|$-point conditional homomorphism density of $H$ in $W$ (see \citet[Definition 2.1]{bhattacharya2021fluctuations}). Then, by the Cauchy-Schwarz inequality and using $|W| \leq 1$,  
    \begin{align*}
    t(H_s\cup H_t,\,W)& \ge \left(\int_{[0,\,1]^{|V(F)|}}t_{V(F)}(x_1,\,x_2,\,\ldots,\,x_{|V(F)|};\,H,\,W)\prod_{a\in V(F)}\text{d}x_a\right)^2\\
    & \geq t(H,\,W)^2 > 0 . 
\end{align*}
Furthermore, note that to create pairs $H_s,\,H_t$ as above, we may first choose the common $|V(F)|$ vertices in $\binom{n}{|V(F)|}$ ways. The rest of the vertices of $H_s$ and $H_t$ can be chosen in $$\binom{n-|V(F)|}{|V(H)|-|V(F)|}\binom{n-|V(H)|}{|V(H)|-|V(F)|}$$ ways. Hence, the number of such pairs is $\Theta_H(n^{2|V(H)|-|V(F)|})$. This completes the proof of the lemma. \hfill $\Box$

\subsection{Proof of Lemma \ref{lm:HWU}}\label{appendix:proofof_varlemma}

We begin by recalling the expression of $T(H, G_n)$ from \eqref{eq:THG} and noting that $T
(H,\,G_n)$ is degenerate if and only if
$$\xi_1:=\Cov[h(U_{i_1},\, U_{i_2},\,\ldots,\,U_{i_{|V(H)|}}),\,h(U_{i_1'},\,U_{i_2'},\,\ldots,\,U_{i_{|V(H)|}'})] = 0 , $$
with $|\{i_1,\,i_2,\,\ldots,\,i_{|V(H)|}\}\cap\{i_1',\,i_2',\,\ldots,\,i_{|V(H)|}'\}|=1.$ We will show that $\xi_1=0$ if and only if the graphon $W$ is $H$-regular. To prove this, observe that
\begin{align*}
    \xi_1&=\Cov[h(U_1,\,U_2,\,\ldots,\,U_{|V(H)|}),\,h(U_{|V(H)|},\,U_{|V(H)|+1},\,\ldots,\,U_{2|V(H)|-1})] \\
    &=\E[h(U_1,\,U_2,\,\ldots,\,U_{|V(H)|})h(U_{|V(H)|},\,U_{|V(H)|+1},\,\ldots,\,U_{2|V(H)|-1})] \\
    &=\sum\limits_{\substack{H'\in\mathscr{G}_H(\{1,\,2,\,\ldots,\,|V(H)|\})\\H''\in\mathscr{G}_H(\{|V(H)|,\,|V(H)|+1,\,\ldots,\,2|V(H)|-1\})}}\left\{\E\left[\prod\limits_{\substack{(i,\,j)\in E(H')\\(k,\,l)\in E(H'')}}W(U_i,\,U_j)W(U_k,\,U_l)\right]-t(H,\,W)^2\right\} . 
    \end{align*}
To further simplify the expression recall Definition \ref{defn:v_join}. Now, applying \cite[Lemma 5.2]{bhattacharya2021fluctuations} gives, 
\begin{align*}
    \xi_1 = \frac{|\mathscr{G}_H|^2}{|V(H)|^2}\left[\sum_{1\le a,\,b\le |V(H)|}t\left(H\bigoplus_{a,\,b} H,\,W\right)-|V(H)|^2t(H,\,W)^2\right].
\end{align*}
Finally, invoking \citet[Lemma 5.4]{bhattacharya2021fluctuations} we conclude that $\xi_1=0$ if and only if $W$ is $H$-regular. Thus, when $W$ is $H$-regular, $T(H, G_n)$ is a first order degenerate $U$-statistics. Hence, by \citet[Section 1.3, Theorem 3]{lee2019ustatistics}, $\Var[T(H,\,G_n)]\lesssim_{H,\,W}\frac{1}{n^2}$. This completes the proof of Lemma \ref{lm:HWU}. \hfill $\Box$

\subsection{Proof of Proposition \ref{threshold2}} 
\label{sec:threshold2pf}

Recall the definition of $T(H, G_n)$ from \eqref{eq:THG}. Since $W$ is not $H$-regular, by Lemma \ref{lm:HWU} $T(H, G_n)$  is a $U$-statistic of order $|V(H)|$ which is non-degenerate. Hence, by \citet[Section 1.3, Theorem 3]{lee2019ustatistics}, 
$\Var[T(H,G_n)] \lesssim_{H, W} \frac{1}{n}$, and recalling \eqref{eq:Delta2Ustatistics}, 
$$\Var[\Delta_2(H,G_n)] \lesssim_{H, W} n^{2|V(H)|-1}\rho_n^{2|E(H)|}.$$ 
Also, recalling \eqref{eq:VarGn} in the case when $W$ is not $H$-regular, 
\begin{align} 
\Var[\Delta(H, G_n)] \asymp_{H, W} \max_{F\subseteq H:|V(F)|\ge 1} n^{2|V(H)|-|V(F)|}\rho_n^{2|E(H)|-|E(F)|}.
\label{eq:var_recall}
\end{align} 
We claim that when $n\rho_n^{m_1(H)}\ll 1,$ the maximum in \eqref{eq:var_recall} is obtained for some subgraph $F\subseteq H$ with $|V(F)|\ge2.$ Indeed, when $n\rho_n^{m_1(H)} \ll 1,$ there exists a subgraph $F\subseteq H$ with $|V(F)|\ge 2$ such that $n^{|V(F)|-1}\rho_n^{|E(F)|}\ll1$ or equivalently, $n^{2|V(H)|-1}\rho_n^{2|E(H)|}\ll n^{2|V(H)|-|V(F)|}\rho_n^{2|E(H)|-|E(F)|}.$ Thus,
$$\frac{\Var[\Delta_2(H,G_n)]}{\Var[\Delta(H, G_n)]}\lesssim_{H,W}\min_{F\subseteq H,|V(F)|\ge2}n^{|V(F)|-1}\rho_n^{|E(F)|}\to0 , $$
since $n\rho_n^{m_1(H)}\ll 1.$ This proves Proposition \ref{threshold2} (1). 

Next, for the second part, from \citet[Section 1.3, Theorem 3]{lee2019ustatistics} we know that 
$$\lim_{n\to\infty}n\Var[T(H,G_n)]=|V(H)|^2\xi_1,$$
where $\xi_1$ is defined in Appendix \ref{appendix:proofof_varlemma}. Note that $\xi_1> 0$, as $W$ is not $H$-regular.  Thus, recalling \eqref{eq:Delta2Ustatistics},
    \begin{equation}
        \Var[\Delta_2(H,G_n)] = \frac{|V(H)|^2}{n}\binom{n}{|V(H)|}^2\rho_n^{2|E(H)|}\xi_1 + o(n^{2|V(H)|-1}\rho_n^{2|E(H)|}) . 
        \label{var-S2}
    \end{equation} 
    Now, continuing from (\ref{eq:var-1}), we have 
    \begin{align}\label{eq:varianceF}
        \Var[\Delta(H,G_n)] & =\sum_{F\subseteq H:|V(F)|=1}\sum_{\substack{H_s,H_t\in\mathcal{H}_n \\ H_s\cap H_t\simeq F}}\Cov[I_s,I_t]  + \sum_{F\subseteq H:|V(F)| \geq 2}\sum_{\stackrel{H_s,H_t\in\mathcal{H}_n}{H_s\cap H_t\simeq F}}\Cov[I_s,I_t] , 
\end{align}
where $F$ is the subgraph of $H$ isomorphic to $H_s\cap H_t=(V(H_s)\cap V(H_t),E(H_s)\cap E(H_t))$. Note that for a fixed subgraph $F\subseteq H,$ the number of pairs $(H_s,H_t) $ having $H_s\cap H_t$ isomorphic to $F$ is $O_H(n^{2|V(H)|-|V(F)|}).$ Hence, the second term in \eqref{eq:varianceF} can be bounded as, 
\begin{align*}
    \sum_{F\subseteq H:|V(F)| \geq 2}\sum_{\substack{H_s,H_t\in\mathcal{H}_n \\ H_s\cap H_t\simeq F}}\Cov[I_s,I_t]&\lesssim_{H,W}\sum_{F\subseteq H:|V(F)| \geq 2}n^{2|V(H)|-|V(F)|}\rho_n^{2|E(H)|-|E(F)|}\\
    &=o(n^{2|V(H)|-1}\rho_n^{2|E(H)|}) , 
\end{align*}
since $n\rho_n^{m_1(H)} \gg 1$ is equivalent to the condition that for all subgraphs $F\subseteq H$ with $|V(F)|\ge 2,$ we have $n^{2|V(H)|-|V(F)|}\rho_n^{2|E(H)|-|E(F)|}\ll n^{2|V(H)|-1}\rho_n^{2|E(H)|}.$ 
To compute the first term in \eqref{eq:varianceF}, first choose the vertex sets of the two copies of $H$ sharing a single vertex. This is equivalent to finding the number of ordered pairs of subsets $(A_1,A_2)$ of $[n]$ such that $|A_1\cap A_2|=1.$ Observe that the number of such ordered pairs is $|V(H)|\binom{n}{|V(H)|}\binom{n-|V(H)|}{|V(H)|-1}.$ Once we have fixed the vertex sets of the two copies, we can trace the steps of the proof of Lemma \ref{lm:HWU} to obtain
    \begin{equation}
        \Var[\Delta(H, G_n)] = |V(H)|\binom{n}{|V(H)|}\binom{n-|V(H)|}{|V(H)|-1}\rho_n^{2|E(H)|}\xi_1 +  o(n^{2|V(H)|-1}\rho_n^{2|E(H)|}) .
        \label{var-S}
    \end{equation}
    from \eqref{eq:varianceF}. Taking the ratio of \eqref{var-S2} and \eqref{var-S} it follows that $\frac{\mathrm{Var}[\Delta_2(H, G_n)]}{\mathrm{Var}[\Delta(H, G_n)]} \rightarrow 1$. Hence, recalling \eqref{eq:vardecomp} the result in Proposition \ref{threshold2} (2) follows. \hfill $\Box$

    \subsection{ Joint Convergence from Marginal Convergence } \label{sec:conditionaljointpf}

In this section we formulate a general result about bivariate convergence given the convergence of one of the marginals and the conditional convergence of the other marginal. This has been used in the proof of \eqref{eq:Deltairregulart}.

\begin{lem} Suppose $\{(X_n, Y_n)\}_{n \geq 1}$ be a collection of random variables and $\{\mathcal F_n\}_{n \geq 1}$ a collection of sigma algebras such $Y_n$ is $\mathcal F_n$-measurable. Assume the following holds: 
\begin{itemize} 
\item $d_{\mathrm{Wass}}(X_n \mid \mathcal F_n,\,Z) \Pto 0$, where $Z \sim N(0,\,1).$
\item $Y_n \dto N(0,\,1).$
\end{itemize}
Then 
$$\begin{pmatrix} 
X_n \\ 
Y_n 
\end{pmatrix} \dto N \left( 
    \begin{pmatrix} 
    0 \\ 
    0
    \end{pmatrix},   \begin{pmatrix} 
    1 & 0 \\ 
    0 &  1
    \end{pmatrix} \right) . $$   
  \label{lm:conditionaljoint}
\end{lem}
\begin{proof} Note that it is enough to show, 
$$\E[e^{isX_n+itY_n}]\to e^{-\frac{s^2+t^2}{2}} , $$
for all $s,\,t\in(-1,\,1).$ Notice that $Y_n$ is $\CF_n$-measurable. This implies,
\begin{align}
    \E[e^{isX_n+itY_n}]&=\E[e^{itY_n}\E[e^{isX_n}\mid\CF_n]]\nonumber\\
    &=\E[e^{itY_n}(\E[e^{isX_n}\mid\CF_n]-e^{-\frac{s^2}{2}})]+e^{-\frac{s^2}{2}}\E[e^{itY_n}] .
    \label{eq:chf}
\end{align}
Since $Y_n \dto N(0,\,1),$ we have $\E[e^{itY_n}]\to e^{-\frac{t^2}{2}}$, for all $t\in(-1,\,1).$ Thus, it is enough to show that the first term in \eqref{eq:chf} converges to 0. To this end, note that,
\begin{align}
    \left|\E[e^{itY_n}\left(\E[e^{isX_n}\mid\CF_n]-e^{-\frac{s^2}{2}}\right)\right|&\le\E\left[\left|e^{itY_n}\right|\left|\E[e^{isX_n}\mid\CF_n]-e^{-\frac{s^2}{2}}\right|\right]\nonumber\\
    &\le\E\left[\left|\E[e^{isX_n}\mid\CF_n]-\E[e^{isZ}]\right|\right].\label{eq:is_diff_bd}
\end{align}
To complete the proof note that,
\begin{align*}
    & \left|\E[e^{isX_n}\mid\CF_n]-\E[e^{isZ}]\right| \\ 
    &\le|\E[\cos(sX_n)\mid\CF_n]-\E[\cos(sZ)]|  + |\E[\sin(sX_n)\mid\CF_n]-\E[\sin(sZ)]| \\
    &\le2d_{\mathrm{Wass}}(X_n \mid \mathcal F_n,\,Z)\Pto0.
\end{align*}
The above convergence follows from Definition \ref{defn:cond_wass}, since $x\mapsto\cos(sx)$ and $x\mapsto\sin(sx)$ are both 1-Lipschitz functions for $|s|<1$. Also, as characteristic functions are uniformly bounded by 1, we have $L_1$ convergence, that is,
$$\E\left[\left|\E[e^{isX_n}\mid\CF_n]-\E[e^{isZ}]\right|\right]\to0.$$
This completes the proof by recalling \eqref{eq:chf} and \eqref{eq:is_diff_bd}.
\end{proof}

\end{document}